\documentclass[11pt]{article}
\usepackage{listings}
\usepackage{pdflscape}
\usepackage{amsfonts}
\usepackage{epsfig}
\usepackage{lscape}
\usepackage{longtable}
\usepackage{amsmath,amssymb,amsthm}
\usepackage{graphicx}
\usepackage{subfigure}
\usepackage{color}
\usepackage{placeins}
\usepackage{url}
\usepackage{cases}
\usepackage{longtable}
\usepackage{supertabular}
\usepackage{multirow}
\newcommand{\tabincell}[2]{\begin{tabular}{@{}#1@{}}#2\end{tabular}}
\usepackage{amsmath}
\allowdisplaybreaks[4]
\usepackage{cite}
\usepackage{algorithmic}
\usepackage{algorithm,float}
\usepackage{booktabs}



\newcommand{\red}[1]{\begin{color}{red}#1\end{color}}

\newtheorem{theorem}{Theorem}

\newtheorem{lemma}{Lemma}
\newtheorem{remark}{Remark}
\newtheorem{definition}{Definition}
\newtheorem{proposition}{Proposition}

\begin{document}

\title{\bf Efficient sparse semismooth Newton methods for the clustered lasso problem\footnotemark[1]}
\author{Meixia Lin\footnotemark[2], \quad Yong-Jin Liu\footnotemark[3], \quad Defeng Sun\footnotemark[4], \quad Kim-Chuan Toh\footnotemark[5]}
\date{May 1, 2019}
\maketitle

\renewcommand{\thefootnote}{\fnsymbol{footnote}}
\footnotetext[1]{{\bf Funding:} The research of Yong-Jin Liu was supported in part by the National Natural Science Foundation of China under Grants 11371255 and 11871153, the research of Defeng Sun was supported in part by a start-up research grant from the Hong Kong Polytechnic University, and the research of Kim-Chuan Toh was supported in part by the Ministry of Education, Singapore, Academic Research Fund under Grant R-146-000-257-112.}
\footnotetext[2]{Department of Mathematics, National University of Singapore, 10 Lower Kent Ridge Road, Singapore ({\tt lin\_meixia@u.nus.edu}).}
\footnotetext[3]{Key Laboratory of Operations Research and Control of Universities in Fujian, College of Mathematics and Computer Science, Fuzhou University, Fuzhou 350116, China ({\tt yjliu@fzu.edu.cn}).}
\footnotetext[4]{Department of Applied Mathematics, The Hong Kong Polytechnic University, Hung Hom, Hong Kong ({\ttfamily defeng.sun@polyu.edu.hk}).}
\footnotetext[5]{Department of Mathematics and Institute of Operations Research and Analytics, National University of Singapore, 10 Lower Kent Ridge Road, Singapore ({\tt mattohkc@nus.edu.sg}).}
\renewcommand{\thefootnote}{\arabic{footnote}}

\begin{abstract}
	We focus on solving the clustered lasso problem, which is a least squares problem with the $\ell_1$-type penalties imposed on both the coefficients and their pairwise differences to learn the group structure of the regression parameters. Here we first reformulate the clustered lasso regularizer as a weighted ordered-lasso regularizer, which is essential in reducing the computational cost from $O(n^2)$ to $O(n\log (n))$. We then propose an inexact semismooth Newton augmented Lagrangian ({\sc Ssnal}) algorithm to solve the clustered lasso problem or its dual via this equivalent formulation, depending on whether the sample size is larger than the dimension of the features.
	An essential component of the {\sc Ssnal} algorithm is the computation of the generalized Jacobian of the proximal mapping of the clustered lasso regularizer. Based on the new formulation, we derive an efficient procedure for its computation.
	Comprehensive results on the global convergence and local linear convergence  of the {\sc Ssnal} algorithm are established. For the purpose of exposition and comparison, we also summarize/design several first-order methods that can be used to solve the problem under consideration, but with the key improvement from the new formulation of the clustered lasso regularizer. As a demonstration of the applicability of our algorithms, numerical experiments on the clustered lasso problem  are performed. The experiments show that the {\sc Ssnal} algorithm substantially outperforms the best alternative algorithm for the clustered lasso problem.
\end{abstract}

\\[5pt]

\section{Introduction}
We consider the minimization problem of the following form:
\begin{equation}\label{given-P}
\min_{x\in \Re^n} \ \displaystyle \Big\{ \frac{1}{2}\|Ax-b\|^2+\beta \|x\|_1 +\rho\sum_{1\leq i<j \leq n} |x_i-x_j|\Big\},
\end{equation}
where $A\in \Re^{m\times n}$, $b\in \Re^m$ are given data and $\beta,\rho> 0 $ are given positive parameters. For $x\in \Re^n$, $\|x\|_1 = \sum_{i=1}^n |x_i|$. Obviously, {\em the optimal solution set of problem \eqref{given-P}, denoted as $\Omega_p$, is nonempty and bounded.} Problems of the form \eqref{given-P} are called the clustered lasso problems, which are motivated by the desire to learn the group structure of the regression parameters $\{x_i\}$ in the statistical context \cite{she2010sparse,petry2011pairwise}. Two types of sparsity are desirable: zero-sparsity and equi-sparsity. The clustered lasso model is proposed with the $\ell_1$-type penalties imposed on both the coefficients and their pairwise differences.

It is worthwhile to mention several other popular models for group sparsity of the regression parameters. The fused lasso model \cite{tibshirani2005sparsity,ye2011split,liu2010efficient} penalizes the differences between the adjacent predictors, which was developed for ordered predictors. The group lasso model \cite{yuan2006model,friedman2010note,jacob2009group} assumes that the grouping of the predictors is known, say from the underlying background, and then penalizes the $\ell_2$-norm of the coefficients within the same predictor group. The OSCAR model \cite{bondell2008simultaneous,zhong2012efficient} penalizes the combination of the $\ell_1$-norm and a pairwise $\ell_{\infty}$-norm for the coefficients. OSCAR is similar to the clustered lasso since it seeks zero-sparsity and equi-sparsity in $\{|x_i|\}$. All these models are extended from the original Lasso model \cite{tibshirani1996regression,wen2010fast,wright2009sparse} to obtain minimal prediction error and also to recover the true underlying specific structure of the model.

The clustered lasso model has been applied in microarray data analysis. Besides, the clustered lasso can be used as a pre-processing step for the fused lasso or the group lasso for uncovering the group structure of the predictors. Researchers have designed some algorithms for solving \eqref{given-P} through reformulating \eqref{given-P} as a constrained Lasso problem by introducing new variables in \cite{she2010sparse,petry2011pairwise,tang2016fused}. Unfortunately, these methods can be hardly applied to the large-scale problems due to huge computational cost.

In real applications, one may need to run the clustered lasso problem \eqref{given-P} many times with different $(\beta,\rho)$ when tuning parameters to get reasonable sparsity structure of the predictors. Therefore, it is important for us to design an efficient and robust algorithm, especially for the high-dimensional and/or high-sample cases. In order to achieve fast convergence, we aim to solve the clustered lasso problem by designing a method which exploits the second order information. Specifically, we will design a semismooth Newton augmented Lagrangian method, which has already been demonstrated to be extremely efficient for Lasso \cite{li2016highly}, fused lasso \cite{li2018efficiently}, group lasso \cite{zhang2017efficient} and OSCAR \cite{luo2018solving}.

The main contributions of our paper can be summarized as follows.
\begin{enumerate}
	\item We reformulate the clustered lasso regularizer as a weighted ordered-lasso regularizer, which is crucial to reducing the cost of computing the regularizer from $O(n^2)$ to $O(n\log(n))$ operations. Based on this reformulation, we are able to compute the proximal mapping of the clustered lasso regularizer by using the pool-adjacent-violators algorithm in $O(n\log(n))$ operations. As far as we are aware of, this is the first time that the proximal mapping of the clustered lasso regularizer is shown to be computable in $O(n\log(n))$ operations.
	
	\item The new formulation is also critical for us to obtain a well-structured generalized Jacobian of the corresponding proximal mapping so that it can be computed explicitly and efficiently with the structure to be mentioned in Section \ref{sub_computable_M}.
	
	\item We propose a semismooth Newton augmented Lagrangian ({\sc Ssnal}) method for solving problem \eqref{given-P} or its dual depending on whether the sample size is larger than the dimension of the features. Since the objective function in \eqref{given-P} is piecewise linear-quadratic, the augmented Lagrangian method ({\sc Alm}) is proved to have the asymptotic superlinear convergence property according to \cite{rockafellar1976augmented,rockafellar1976monotone,li2016highly}. For the {\sc Alm} subproblem, we employ a semismooth Newton method that exploits the second-order sparsity of the generalized Jacobian of the proximal mapping of the clustered lasso regularizer to get fast superlinear or even quadratic convergence.
	
	\item As the first-order methods have been very popular in solving various lasso-type problems in recent years,  we summarize two first-order algorithms which can be used to solve problem \eqref{given-P}. The computation of the key projection step is highly improved due to our new formulation of the clustered lasso regularizer.
	
	\item We conduct comprehensive numerical experiments to demonstrate the efficiency and robustness of the {\sc Ssnal} method against different parameter settings. We also demonstrate the superior performance of our algorithm over other first-order methods for large-scale instances with $n\gg m$.
\end{enumerate}

The remaining parts of this paper are organized as follows. The next section is devoted to computing and analyzing the proximal mapping of the clustered lasso regularizer and its generalized Jacobian. In Sections \ref{sect:ssnal} and \ref{sect:ssnalP}, we develop semismooth Newton based augmented Lagrangian algorithms to solve the clustered lasso problem and its dual problem, respectively. We employ various numerical techniques to efficiently exploit the second-order sparsity and special structure of the generalized Jacobian when implementing the {\sc Ssnal} algorithms. For the purpose of evaluating the efficiency of our {\sc Ssnal} algorithms, in  Section \ref{sect:FOMs} we summarize two first-order algorithms which are conducive for solving the general problem \eqref{given-P}. By using the proposed proximal mapping of the clustered lasso regularizer to be given in Section \ref{subsection_prox_mapping}, one can compute the key projection step in these two  first-order methods efficiently in $O(n\log (n))$ operations. This is already a significant improvement over the current methods in \cite{she2010sparse,petry2011pairwise,tang2016fused}, which require $O(n^2)$ to just evaluate the clustered lasso regularizer.
The numerical performance of our {\sc Ssnal} algorithms for the clustered lasso problems on large scale real data and synthetic data against other state-of-the-art algorithms are presented in Section \ref{sect:NumRes}. We conclude our paper in the final section.

\bigskip
\noindent{\bf Notation.}
Throughout the paper, we use ``${\rm diag}(X)$" to denote the vector consisting of the diagonal entries of the matrix $X$ and ``${\rm Diag}(x)$" to denote the diagonal matrix whose diagonal is given by the vector $x$. We denote by $I_n$, ${\bf O}_n$, and ${\bf E}_n$ the $n\times n$ identity matrix, the $n\times n$ zero matrix, and the $n\times n$ matrix of all ones, respectively. For given matrix $C$, we also use $C^{\dagger}$ to represent its Moore-Penrose inverse. As usual, $f^*$ is the Fenchel conjugate  of an arbitrary function $f$.

\section{Computing the proximal mapping of the clustered lasso regularizer and its generalized Jacobian}\label{mapping_and_jacobian}

For convenience, we denote the clustered lasso regularizer in \eqref{given-P} by
\begin{equation*}
p(x)=\beta \|x\|_1+\rho \sum_{1\leq i<j \leq n} |x_i-x_j|,\ \forall x\in\Re^n.
\end{equation*}
Let $f:\Re^n \rightarrow (-\infty,\infty]$ be  any given proper closed convex function. Then, the proximal mapping ${\rm Prox}_{f}(\cdot)$ of $f$ is defined as
\[
{\rm Prox}_f(y)=\underset{x\in\Re^n}{\rm argmin}\Big\{  \frac{1}{2}\|x-y\|^2+f(x) \Big\},\ \forall y\in \Re^n.
\]
We have the following important Moreau's identity:
\[
{\rm Prox}_{tf}(x)+t{\rm Prox}_{f^*/t}(x/t)=x,
\]
where $t>0$ is a given parameter.

In this section, we shall develop some useful results on calculating the proximal mapping of the clustered lasso regularizer $p(\cdot)$ and the corresponding generalized Jacobian.

\subsection{The computation of the proximal mapping ${\rm Prox}_p(\cdot)$} \label{subsection_prox_mapping}

Denote
\[
S_{\rho}(y):=\underset{x\in\Re^n}{\rm argmin} \ \displaystyle\Big\{ \frac{1}{2}\|x-y\|^2 + \rho \sum_{1\leq i<j \leq n} |x_i -x_j|   \Big\},\ \forall y\in \Re^n,
\]
and ${\cal D}=\{x\in \Re^n \mid Bx\geq 0\}$, where $B$ is a matrix such that $Bx=[x_1-x_2;\cdots;x_{n-1}-x_n]\in \Re^{n-1}$.

We shall reformulate the clustered lasso regularizer as a weighted ordered-lasso regularizer, which enables us to reduce the cost of computing the regularizer from $O(n^2)$ to $O(n\log(n))$ operations.
For any $x\in \Re^n$, we define $x^\downarrow$ to be the vector whose components are those of $x$ sorted in a non-increasing order, i.e. $x^\downarrow_1\geq x^\downarrow_2\geq \cdots \geq x^\downarrow_n$.

\begin{proposition}\label{reformulate_g}
	Let $x\in \Re^n$ be an arbitrarily given vector.  Then it  holds that
	\[
	g(x) \;:=\;\sum_{1\leq i < j \leq n} |x_i -x_j| \; =  \;   \langle w,x^\downarrow \rangle,
	\]
	where the vector $w\in\Re^n$ is defined by
	\begin{equation}\label{eq:w}
	w_k= n-2k+1, \ k=1,\cdots,n.
	\end{equation}
\end{proposition}
\begin{proof}
	By noting that $g(x) = g(Px)$ for any permutation matrix $P$, one has that
	\begin{equation*}
	\begin{split}
	g(x)
	&= \sum_{1\leq i < j \leq n} |x^\downarrow_i -x^\downarrow_j|\ = \ \sum_{1\leq i < j \leq n} (x^\downarrow_i -x^\downarrow_j)
	\\
	&=  \sum_{i=1}^{n-1} (n-i) x^\downarrow_i - \sum_{j=2}^n (j-1) x^\downarrow_j=\sum_{k=1}^n  (n-2k+1) x^\downarrow_k,
	\end{split}
	\end{equation*}
	which completes the proof.
\end{proof}
\begin{remark}
	As a side note, the result in Proposition \ref{reformulate_g} is {\bf not} valid for a nonuniformly weighted sum.
\end{remark}

The next proposition shows that if a vector $y\in\Re^n$ is sorted in a non-increasing order, $S_{\rho}(y)$ can be computed by a single metric projection onto ${\cal D}$.

\begin{proposition} \label{sorty_S}
	Suppose that $y\in\Re^n$ is given such that $y_1\geq y_2 \geq \cdots\geq y_n$. Then it holds
	\begin{equation*}
	S_{\rho}(y)=\Pi_{\cal D}(y-\rho w),
	\end{equation*}
	where $w\in\Re^n$ is given in (\ref{eq:w}). The metric projection onto ${\cal D}$ can be computed via the pool-adjacent-violators algorithm \cite{best1990active}.
\end{proposition}
\begin{proof}
	Let $g(\cdot)$ be defined in Proposition \ref{reformulate_g}.
	We first note that $g(x) = g(Px)$ for any permutation matrix $P$ and $x\in \Re^n$. For convenience, let $x^* = S_{\rho}(y)$. Next we show that the components of $x^*$ must be arranged in a non-increasing order. Suppose on the contrary that there exists $i<j$ such that $x_i^* < x_j^*$. We define $\bar{x}\in\Re^n$ by $\bar{x}_i = x_j^*$, $\bar{x}_j = x_i^*$, $\bar{x}_k = x_k^*$ for all $k\not = i,j$. Then, we derive that
	\begin{equation*}
	\begin{split}
	&\frac{1}{2}\|x^* - y\|^2 + \rho g(x^*)- \Big(\frac{1}{2}\|\bar{x}-y\|^2 + \rho g(\bar{x})\Big)
	\\
	&=\frac{1}{2} \Big( (x^*_i - y_i)^2 + (x^*_j - y_j)^2 - (x^*_j - y_i)^2 - (x^*_i - y_j)^2 \Big)
	=(x_j^*-x_i^*) ( y_i - y_j) \geq 0,
	\end{split}
	\end{equation*}
	which implies that $\bar{x}$ is also a minimizer. By the uniqueness of the minimizer, we have that $\bar{x}=x^*$ and hence
	$x^*_j=\bar{x}_i=x^*_i$, which is a contradiction. Hence, we obtain that
	\begin{equation*}
	\begin{split}
	x^* &= \underset{x\in \Re^n}{\rm argmin} \ \displaystyle \Big\{ \frac{1}{2}\|x-y\|^2 + \rho g(x) \mid x_1 \geq x_2 \geq \cdots \geq x_n \Big\}
	\\
	&= \underset{x\in \Re^n}{\rm argmin} \  \displaystyle \Big\{ \frac{1}{2}\|x-y\|^2 + \rho \langle w,x\rangle \mid x_1 \geq x_2 \geq \cdots \geq x_n \Big\}
	\\
	&= \underset{x\in \Re^n}{\rm argmin} \  \displaystyle \Big\{ \frac{1}{2}\|x-(y-\rho w)\|^2
	\mid x_1 \geq x_2 \geq \cdots \geq x_n \Big\}
	\\
	&=\Pi_{\cal D}(y-\rho w).
	\end{split}
	\end{equation*}
	The proof is complete.
\end{proof}

Combining Proposition \ref{reformulate_g} with Proposition \ref{sorty_S}, we can get an explicit formula for $S_{\rho}(\cdot)$. Let $y\in\Re^n$ be given. Then  there exists a permutation matrix $P_y\in \Re^{n\times n}$ such that
$\tilde{y}=P_y y$ and $\tilde{y}_1 \geq \tilde{y}_2\geq\cdots\geq \tilde{y}_n$. Thus,
\begin{equation*}
S_{\rho}(y)= P_y^{T} S_{\rho}(\tilde{y})= P_y^{T} \Pi_{\cal D}(\tilde{y}-\rho w)= P_y^T \Pi_{\cal D}(P_yy-\rho w).
\end{equation*}

Next we recall an important result on computing ${\rm Prox}_p(\cdot)$, which comes from \cite[Corollary 4]{yu2013decomposing}.

\begin{proposition}\label{prop_proximal_mapping}
	Let $y\in \Re^n$ be given. Then, we have that
	\begin{equation*}
	{\rm Prox}_{p}(y)= {\rm Prox}_{\beta\|\cdot\|_1} (S_{\rho}(y))= {\rm sign}(S_{\rho}(y))\circ\max(|S_{\rho}(y)|-\beta,0),
	\end{equation*}
	where ``$\circ$" denotes the Hadamard product.
\end{proposition}

The above proposition states that the proximal mapping of the clustered lasso regularizer can be decomposed into the composition of the proximal mapping of $\beta\|\cdot\|_1$ and the proximal mapping of $\rho g(\cdot)$.

\subsection{The computation of the generalized Jacobian of ${\rm Prox}_p(\cdot)$}
We first present some results on the generalized HS-Jacobian of $\Pi_{\cal D}$, which can be obtained directly from the previous work in \cite{han1997newton}, wherein
Han and Sun  constructed theoretically computable generalized Jacobian of the metric projector over a polyhedral set. Recently, Li et al. \cite{li2018efficiently} further derived an efficient formula for computing a special HS-Jacobian of the solution mapping of a parametric strongly convex quadratic programming.
In this section, we will adapt the ideas in \cite{li2018efficiently} to efficiently compute the generalized Jacobian of $\Pi_{\cal D}(\cdot)$.

Since $\Pi_{\cal D}$ is the metric projection onto the nonempty polyhedral set ${\cal D}$, for any given $y\in \Re^n$, there exists a multiplier $\lambda\in \Re^{n-1}$ such that the following KKT system holds:
\begin{equation}\label{Pi_D_kkt}
\left\{
\begin{array}{l}
\Pi_{\cal D}(y)-y+B^T\lambda=0,\\[5pt]
B\Pi_{\cal D}(y)\geq 0, \ \lambda\leq 0,\\[5pt]
\lambda^T B\Pi_{\cal D}(y)=0.
\end{array}\right.
\end{equation}
Let ${\cal M}_{\cal D}(y):=\{\lambda \in \Re^{n-1}\mid (y,\lambda)\ {\rm satisfies} \ \eqref{Pi_D_kkt}\}$. Since ${\cal M}_{\cal D}(y)$ is a nonempty polyhedral convex set which contains no lines, it has at least one extreme point \cite[Corollary 18.5.3]{rockafellar2015convex}.
Denote the active index set by
\begin{equation}\label{activeset}
{\cal I}_{\cal D}(y):=\{i\mid B_i \Pi_{\cal D}(y)=0,i=1,2,\cdots,n-1\},
\end{equation}
where $B_i$ is the $i$-th row of $B$. Define a collection of index subsets of $\{1,\cdots,n-1\}$ as follows
\begin{equation*}
{\cal K}_{\cal D}(y):=\Big\{K\mid \exists \lambda  \in {\cal M}_{\cal D}(y)  \ {\rm s.t.} \ {\rm supp}(\lambda)\subseteq K \subseteq {\cal I}_{\cal D}(y),B_K \mbox{ is of full row rank}\Big\},
\end{equation*}
where ${\rm supp}(\lambda)$ denotes the support of $\lambda$ and $B_K$ is the matrix consisting of the rows of $B$ indexed by $K$. It should be noted that ${\cal K}_{\cal D}(y)$ is nonempty due to the existence of an extreme point of ${\cal M}_{\cal D}(y)$ as stated in \cite{han1997newton}. Han and Sun in \cite{han1997newton} introduced the following multifunction ${\cal Q}_{\cal D}:\Re^n\rightrightarrows\Re^{n\times n}$ defined by
\begin{equation*}
{\cal Q}_{\cal D}(y):=\Big\{\widehat{Q}\in \Re^{n\times n}\mid \widehat{Q}=I_n-B_K^T(B_K B_K^T)^{-1}B_K,K\in {\cal K}_{\cal D}(y)\Big\},
\end{equation*}
which is called the generalized HS-Jacobian of $\Pi_{\cal D}$ at $y$. From \cite[Proposition 1 \& Theorem 1]{li2017efficient}, we can readily get the following proposition, whose proof is omitted for brevity.
\begin{proposition}\label{pro_QD}
	For any $y\in \Re^n$, there exists a neighborhood ${\cal Y}$ of $y$ such that
	\[
	{\cal K}_{\cal D}(u)\subseteq{\cal K}_{\cal D}(y),\ {\cal Q}_{\cal D}(u)\subseteq {\cal Q}_{\cal D}(y), \ \forall u\in {\cal Y},
	\]
	and
	\[
	\Pi_{\cal D}(u)=\Pi_{\cal D}(y)+ \widehat{Q}(u-y),\ \forall \widehat{Q}\in {\cal Q}_{\cal D}(u).
	\]
	Thus, $\partial_{B}\Pi_{\cal D}(y)\subseteq  {\cal Q}_{\cal D}(y)$, where $\partial_{B}\Pi_{\cal D}(y)$ is the B-subdifferential of $\Pi_{\cal D}$ at $y$. In particular,  $\widehat{Q}_{{\cal D},0}(y) \in {\cal Q}_{\cal D}(y)$, where
	\[
	\widehat{Q}_{{\cal D},0}(y):=I_n-B_{{\cal I}_{\cal D}(y)}^T \left(B_{{\cal I}_{\cal D}(y)} B_{{\cal I}_{\cal D}(y)}^T\right)^{\dagger}B_{{\cal I}_{\cal D}(y)}.
	\]
\end{proposition}

Next, we propose a simple and useful result for our further discussions. Given $y\in \Re^n$ and $K\subseteq \{1,\cdots,n-1\}$, we provide an alternative way to compute $ I_n-B_K^T(B_K B_K^T)^{\dagger}B_K$. Let $\Sigma_{K}={\rm Diag}(\sigma_{K})\in \Re^{(n-1)\times(n-1)}$ be defined by
\begin{equation*}
(\sigma_{K})_i=\left\{
\begin{array}{ll}
1, & \mbox{if $i\in K$}\\[5pt]
0, & \mbox{otherwise}
\end{array}\right.\quad \mbox{for $i=1,2,\cdots,n-1$.}
\end{equation*}
By using the fact that there exists a permutation matrix $P_K$ such that
\begin{align*}
\left[ \begin{array}{c}
B_{K}\\[2mm]
\textbf{0}
\end{array}\right]_{(n-1)\times n}=P_K\Sigma_K B=\left[ \begin{array}{ll}
I_{|K|} & \textbf{0}\\[5pt]
\textbf{0}  & \textbf{0}
\end{array}\right]_{(n-1)\times (n-1)} P_K B,
\end{align*}
one can easily prove the following proposition, which will be used later.
\begin{proposition}\label{lemma_inv}
	It holds that
	\[I_n-B_{K}^T(B_{K} B_{K}^T)^{\dagger}B_{K}=I_n-B^T(\Sigma_K B B^T \Sigma_K)^{\dagger}B.\]
\end{proposition}

For convenience, we state Lemma \ref{lemma_T} and Proposition \ref{prop_gamma} below that are discussed in \cite[Lemma 2 \& Proposition 6]{li2018efficiently}.
For $2\leq j \leq n$, we define the linear mapping $ {\bf B}_j:\Re^j \rightarrow \Re^{j-1}$ such that ${\bf B}_j x=[x_1-x_2;\cdots;x_{j-1}-x_j]$, $ \forall x\in \Re^j$. With this notation, we can write $B={\bf B}_n$.

\begin{lemma}\label{lemma_T}
	For $2\leq j \leq n$, it holds that
	\[T_j:=I_j-{\bf B}_j^T({\bf B}_j {\bf B}_j^T)^{-1}{\bf B}_j=\frac{1}{j}{\bf E}_j.\]
\end{lemma}

\begin{proposition}\label{prop_gamma}
	Let $\Sigma\in \Re^{(n-1)\times(n-1)}$ be an N-block diagonal matrix with $\Sigma={\rm Diag}( \Lambda_1,\cdots,\Lambda_N)$, where for $i=1,\cdots,N$, $\Lambda_i$ is either ${\bf O}_{n_i}$ or $I_{n_i}$, and any two consecutive blocks are not of the same type. Denote $J:=\{j\mid \Lambda_j=I_{n_j},j=1,\cdots,N\}.$ Then, it holds that
	\[
	\Gamma:=I_n-B^T(\Sigma B B^T \Sigma)^{\dagger}B={\rm Diag}(\Gamma_1,\cdots,\Gamma_N),
	\]
	where for $i=1,\cdots,N$,
	\begin{equation*}
	\Gamma_i=\left\{
	\begin{array}{ll}
	\frac{1}{n_i+1}{\bf E}_{n_i+1}, & \mbox{if $i\in J$},\\[5pt]
	I_{n_i}, & \mbox{if $i \notin J $ and $i\in \{1,N\}$},\\[5pt]
	I_{n_i-1}, & \mbox{otherwise,}
	\end{array}\right.
	\end{equation*}
	with the convention $I_0=\emptyset$. Moreover, $\Gamma=H+UU^T=H+U_J U_J^T$, where $H\in \Re^{n\times n}$ is an N-block diagonal matrix given by $H={\rm Diag}(\Upsilon_1,\cdots,\Upsilon_N)$ with
	\begin{align*}
	\Upsilon_i\;=\;\left\{\begin{array}{ll}
	{\bf O}_{n_i+1}, & \mbox{if $i\in J$},\\[2mm]
	I_{n_i}, & \mbox{if $i \notin J $ and $i\in \{1,N\}$},\\[2mm]
	I_{n_i-1}, & \mbox{otherwise.}
	\end{array}\right.
	\end{align*}
	Here the $(k,j)$-th entry of the matrix $U\in\Re^{n\times N}$ is given by
	\begin{equation*}
	U_{k,j}=\left\{
	\begin{array}{ll}
	\frac{1}{\sqrt{n_j+1}}, & \mbox{if $\sum_{t=1}^{j-1}n_t+1\leq k\leq \sum_{t=1}^j n_t+1$, and $j\in J$},\\[5pt]
	0, & \mbox{otherwise,}
	\end{array}\right.
	\end{equation*}
	and $U_J$ consists of the nonzero columns of $U$, i.e., the columns whose indices are in $J$.
\end{proposition}

Based on the above preliminaries, we define the multifunction ${\cal Q}_{S_{\rho}}:\Re^n\rightrightarrows\Re^{n\times n}$ by
\[
{\cal Q}_{S_{\rho}}(y):=\Big\{Q\in \Re^{n\times n}\mid Q=P_y^T\widehat{Q}P_y, \ \widehat{Q}\in {\cal Q}_{\cal D}(P_y y-\rho w)\Big\}.
\]
The following proposition shows that ${\cal Q}_{S_{\rho}}(y)$ can be viewed as the generalized Jacobian of $S_{\rho}(\cdot)$ at $y$.
\begin{proposition}\label{pro_s}
	For any $y\in \Re^n$, there exists a neighborhood ${\cal Y}$ of $y$ such that for all $u\in {\cal Y}$,
	\[
	{\cal K}_{\cal D}(P_y u-\rho w)\subseteq {\cal K}_{\cal D}(P_y y-\rho w),\  {\cal Q}_{\cal D}(P_y u -\rho w )\subseteq {\cal Q}_{\cal D}(P_y y -\rho w), \  {\cal Q}_{S_{\rho}}(u) \subseteq {\cal Q}_{S_{\rho}}(y)
	\]
	and
	\begin{equation*}
	\left\{
	\begin{aligned}
	&\Pi_{\cal D}(P_y u-\rho w)=\Pi_{\cal D}(P_y y -\rho w)+ \widehat{Q} P_y(u-y),\ \forall \widehat{Q}\in {\cal Q}_{\cal D}(P_y u-\rho w),\\
	& S_{\rho}(u)\;=\;S_{\rho}(y)+Q(u-y),\ \forall Q\in {\cal Q}_{S_{\rho}}(u).
	\end{aligned}
	\right.
	\end{equation*}
\end{proposition}
\begin{proof}
	The desired results can be easily derived from Proposition \ref{pro_QD} together with simple manipulations.
\end{proof}

Define the multifunction ${\cal M}:\Re^n\rightrightarrows\Re^{n\times n}$ by
\begin{equation}\label{def_M}
{\cal M}(y):=\Big\{M\in {\cal S}^{n}\mid M=\Theta Q, \ \Theta\in  {\partial}_B{\rm Prox}_{\beta\|\cdot\|_1}(S_{\rho}(y)),\ Q\in {\cal Q}_{S_{\rho}}(y)\Big\},
\end{equation}
where the B-subdifferential of ${\rm Prox}_{\beta \|\cdot\|_1}(\cdot)$ at $\eta\in\Re^n$ is given by
\begin{equation*}
{\partial} _B{\rm Prox}_{\beta \|\cdot\|_1}(\eta) =
\left\{ {\rm Diag}(q) \left |
\begin{array}{ll}
q_i=0 & \mbox{if $|\eta_i| < \beta$}
\\
q_i\in \{0,1\}  & \mbox{if $|\eta_i| = \beta$}
\\
q_i =1 & \mbox{otherwise}
\end{array}\right.
\right\}.
\end{equation*}

We can view ${\cal M}(y)$ as the generalized Jacobian of ${\rm Prox}_p(\cdot)$ at $y$. The reason is shown in the following theorem, which is similar to what was done in \cite[Theorem 1]{li2018efficiently} for the fused lasso proximal mapping.
\begin{theorem}\label{theo_M}
	Let $\beta,\rho> 0$ and $y\in \Re^n$ be given. Then, the multifunction ${\cal M}$ is nonempty, compact, and upper-semicontinuous. For any $M\in{\cal M}(y)$, $M$ and $I-M$ are both symmetric and positive semidefinite. Moreover, there exists a neighborhood $\cal{Y}$ of $y$ such that for all $u\in \cal{Y}$,
	\begin{equation}\label{theo1_eq}
	{\rm Prox}_p(u)-{\rm Prox}_p(y)-M(u-y)=0,\ \forall M\in {\cal M}({u}).
	\end{equation}
\end{theorem}
\begin{proof}
	From the definition of ${\cal M}$, we easily see that it is nonempty and compact. We know that ${\partial}_B{\rm Prox}_{\beta\|\cdot\|_1}(\cdot)$ is upper semicontinuous, which, together with the property on $S_{\rho}(\cdot)$ in Proposition \ref{pro_s}, implies that ${\cal M}$ is upper-semicontinuous. In addition, by noting that ${\rm Prox}_{\beta\|\cdot\|_1}(\cdot)$ is piecewise affine, we have that \eqref{theo1_eq} follows from \cite[Theorem 7.5.17]{facchinei2007finite}.
	
	Next we only need to prove that any $M\in {\cal M}(y)$ is symmetric and positive semidefinite. The symmetry follows directly from the definition. From \eqref{def_M} and Lemma \ref{lemma_inv}, one knows that for any $M\in{\cal M}(y)$, there exists a $0$-$1$ diagonal matrix $\Theta\in {\partial}_B {\rm Prox}_{\beta\|\cdot\|_1}(S_{\rho}(y))$ and $K\in {\cal K}_{\cal D}(P_y y -\rho w)$ such that
	\begin{equation*}
	\begin{split}
	M &=\Theta[P_y^T( I_n-B_K^T(B_K B_K^T)^{-1}B_K ) P_y]\\
	&= \Theta P_y^T (I_n-B^T(\Sigma_K B B^T \Sigma_K)^{\dagger}B ) P_y.
	\end{split}
	\end{equation*}
	Since $\Sigma_K\in\Re^{(n-1)\times(n-1)}$ is an $N$-block diagonal matrix with
	\[
	\Sigma_K={\rm Diag}\{\Lambda_1,\cdots,\Lambda_N\},
	\]
	where for $i=1,\cdots,N$, $\Lambda_i$ is either ${\bf O}_{n_i}$ or $I_{n_i}$, and any two consecutive blocks are not of the same type. Denote $J:=\{j\mid \Lambda_j=I_{n_j},j=1,\cdots,N\}.$ It then follows from Proposition \ref{prop_gamma} that
	\[
	M=\Theta P_y ^T\Gamma P_y,
	\]
	where $\Gamma={\rm Diag}(\Gamma_1,\cdots,\Gamma_N)$ is defined as in Proposition \ref{prop_gamma}. Define $\widetilde{\Theta}\in\Re^{n\times n}$ as
	\[
	\widetilde{\Theta}=P_y \Theta P_y^T={\rm Diag } (P_y{\rm diag}(\Theta)),
	\]
	which is also a $0$-$1$ diagonal matrix. Thus,
	\begin{equation*}
	M = P_y^T\widetilde{\Theta} P_y  P_y ^T\Gamma P_y = P_y^T \widetilde{\Theta} \Gamma P_y= P_y^T (\widetilde{\Theta} \Gamma) P_y .
	\end{equation*}
	In order to prove that $M$ is positive semidefinite, it suffices to show that $\widetilde{\Theta}\Gamma$ is positive semidefinite.
	Note that $\widetilde{\Theta}$ can be decomposed as
	$\widetilde{\Theta}={\rm Diag}(\widetilde{\Theta}_1,\cdots,\widetilde{\Theta}_N)$ and hence $\widetilde{\Theta} \Gamma = {\rm Diag}(\widetilde{\Theta}_1\Gamma_1,\cdots,\widetilde{\Theta}_N\Gamma_N) $, we only need to prove that for all $j=1,\cdots,N$, $\widetilde{\Theta}_j\Gamma_j$ is positive semidefinite. When $\Gamma_j$ is an identity matrix, it is obvious that  $\widetilde{\Theta}_j\Gamma_j=\widetilde{\Theta}_j$ and hence $\widetilde{\Theta}_j\Gamma_j$ is positive semidefinite. When $\Gamma_j$ is not an identity matrix but of the form $\Gamma_j=\frac{1}{n_j+1}{\bf E}_{n_j+1}$ from Proposition \ref{prop_gamma}, then we have
	\[
	\Bigg\{\sum_{t=1}^{j-1}n_t +1, \sum_{t=1}^{j-1}n_t +2, \cdots, \sum_{t=1}^{j}n_t \Bigg\}\subseteq K \subseteq {\cal I}_{\cal D}(P_y y -\rho w),
	\]
	which means that
	\[
	\big( \Pi_{\cal D}(P_y y-\rho w)\big)_{i}=\big( \Pi_{\cal D}(P_y y -\rho w)\big)_{i+1}, \ \forall i\in \Bigg\{\sum_{t=1}^{j-1}n_t +1,\cdots,\sum_{t=1}^{j}n_t\Bigg\} .
	\]
	As one can see no matter what value $|\big( \Pi_{\cal D}(P_y y -\rho w)\big)_{\sum_{t=1}^{j-1}n_t +1}|$ takes, ${\rm diag}(\widetilde{\Theta}_j)$ should be all ones or all zeros, otherwise it will contradict the fact that $\widetilde{\Theta}_j\Gamma_j$ is symmetric. That is to say,
	\[
	\widetilde{\Theta}_j=\textbf{O}_{n_j+1} \quad \mbox{or} \quad I_{n_j+1}.
	\]
	Thus, $\widetilde{\Theta}_j\Gamma_j=\textbf{O}_{n_j+1}$ or $\frac{1}{n_j+1}{\bf E}_{n_j+1}$, which is obviously positive semidefinite.
	
	For the case of $I-M$, we have that
	\[
	I-M=I-P_y^T (\widetilde{\Theta} \Gamma) P_y=P_y^T (I-\widetilde{\Theta} \Gamma) P_y.
	\]
	From the previous derivation, we can see that $0\preceq \widetilde{\Theta} \Gamma \preceq I$, which yields that $I-M$ is positive semidefinite. This completes the proof.
\end{proof}

For later purpose, we recall the concept of semismoothness introduced in \cite{mifflin1977semismooth,qi1993nonsmooth,kummer1988newton,sun2002semismooth}.
\begin{definition}
	Let $f:{\cal O}\subseteq \Re^n\to\Re^m $ be a locally Lipschitz continuous function on the open set ${\cal O}$ and ${\cal K}:{\cal O}\rightrightarrows\Re^{m\times n}$ be a nonempty, compact valued and upper-semicontinuous multifunction. We say that $f$ is {\rm semismooth} at $x\in{\cal O}$ with respect to the multifunction ${\cal K}$ if
	(i) $f$ is directionally differentiable at $x$; and
	(ii) for any $\Delta x\in \Re^n$ and $V\in{\cal K}(x+\Delta x)$ with $\Delta x\to 0$,
	\begin{equation}\label{ss}
	f(x+\Delta x)-f(x)-V(\Delta x)=o(\|\Delta x\|).
	\end{equation}
	Furthermore, if \eqref{ss} is replaced by
	\begin{equation}\label{strongly_ss}
	f(x+\Delta x)-f(x)-V(\Delta x)=O(\|\Delta x\|^{1+\gamma}),
	\end{equation}
	where $\gamma>0$ is a constant, then $f$ is said to be {\rm $\gamma$-order (strongly if $\gamma=1$) semismooth} at $x$ with respect to ${\cal K}$. We say that $f$ is a semismooth function on ${\cal O}$ with respect to ${\cal K}$ if it is semismooth everywhere in ${\cal O}$ with respect to ${\cal K}$.
\end{definition}

\begin{remark}\label{M_order_semi}
	Since ${\rm Prox}_p(\cdot)$ is a Lipschitz continuous piecewise affine function, it follows from \cite[Lemma 4.6.1]{facchinei2007finite} that it is directionally differentiable at any point. Combining with Theorem \ref{theo_M}, we conclude that for any arbitrary constant $\gamma>0$, ${\rm Prox}_p(\cdot)$ is $\gamma$-order semismooth on $\Re^n$ with respect to ${\cal M}$.
\end{remark}

\subsection{Finding a computable element in ${\cal M}(y)$}\label{sub_computable_M}

In order for the multifunction that we defined in \eqref{def_M} to be useful in designing algorithms for problem \eqref{given-P}, we need to construct at least one computable element explicitly in ${\cal M}(y)$ for any given $y\in \Re^n$. Let $\Sigma={\rm Diag}(\sigma)\in \Re^{(n-1)\times(n-1)}$ be defined as
\begin{equation}\label{def_sigma}
\sigma_i=\left\{
\begin{array}{ll}
1, & \mbox{if $i\in {\cal I}_{\cal D}(P_y y -\rho w)$,}\\[5pt]
0, & \mbox{otherwise,}
\end{array}\right.\quad \mbox{for $i=1,2,\cdots,n-1$,}
\end{equation}
where ${\cal I}_{\cal D}(\cdot)$ is defined in \eqref{activeset}, and $\Theta={\rm Diag}(\theta)\in \Re^{n\times n}$ be defined as
\begin{equation}\label{def_theta}
\theta_i = \left \{
\begin{array}{ll}
0, & \mbox{if $|S_{\rho}(y)|_{i} \leq \beta $}\\[2mm]
1, & \mbox{otherwise}
\end{array}\right.
\quad \mbox{for $i=1,2,\cdots,n$.}
\end{equation}
From Proposition \ref{pro_QD} and Proposition \ref{lemma_inv}, we have that $M\in {\cal M}(y)$, which is given by
\begin{equation*}
\begin{split}
M & =\Theta P_y^T (I_n-B_{{\cal I}_{\cal D}(P_y y -\rho w)}^T(B_{{\cal I}_{\cal D}(P_y y -\rho w)} B_{{\cal I}_{\cal D}(P_y y -\rho w)}^T)^{\dagger}B_{{\cal I}_{\cal D}(P_y y -\rho w)})P_y\\
& =\Theta P_y^T (I_n-B^T(\Sigma B B^T \Sigma)^{\dagger}B)P_y.
\end{split}
\end{equation*}
Then we can apply Proposition \ref{prop_gamma} to compute $M$ explicitly.

\section{A semismooth Newton augmented Lagrangian method for the dual problem} \label{sect:ssnal}

The primal form of our concerned problem \eqref{given-P} can be written as
\begin{equation}\tag{P}\label{P}
\min_{x\in \Re^n} \ \  \displaystyle\Big\{f(x):=\frac{1}{2} \|Ax-b\|^2+p(x)\Big\},
\end{equation}
and the dual of \eqref{P} admits the following equivalent minimization form
\begin{equation}\tag{D}\label{D}
\min_{\xi\in\Re^m,u \in \Re^n} \ \displaystyle\Big\{  \frac{1}{2}\|\xi\|^2+ \langle b,\xi \rangle+p^*(u)  \mid  A^T \xi+u=0 \Big\}.
\end{equation}
The Lagrangian function associated with \eqref{D} is defined by
\[
l(\xi,u;x):= \frac{1}{2}\|\xi\|^2+ \langle b,\xi \rangle +p^*(u) -\langle x,A^T \xi+u\rangle.
\]
Let $\sigma>0$ be given. Then, the corresponding augmented Lagrangian function is given by
\begin{equation*}
{\cal L}_\sigma( \xi,u; x):= l(\xi,u;x)+\frac{\sigma}{2}\|A^T \xi+u\|^2.
\end{equation*}

\subsection{A semismooth Newton augmented Lagrangian method for \eqref{D}}

We denote the whole algorithm as {\sc Ssnal} since a semismooth Newton method ({\sc Ssn}) is used in solving the subproblem of the inexact augmented Lagrangian method ({\sc Alm}) \cite{rockafellar1976augmented}. We briefly describe the {\sc Ssnal} algorithm as follows.

\begin{algorithm}
	\caption{\small {\bf: ({\sc Ssnal}) A semismooth Newton augmented Lagrangian method for \eqref{D} }}
	\hspace*{0.02in} \raggedright {\bf Input:} $\sigma_0 >0$, $(\xi^0,u^0,x^0)\in\Re^m\times \Re^n \times \Re^n$ and $k=0$.\\
	
	\begin{algorithmic}[1]
		\STATE  Approximately compute
		\begin{equation}\label{d:alm-sub}
		\xi^{k+1} \approx \underset{\xi\in\Re^m}{\rm argmin} \Big\{\psi_k(\xi):=\inf_u {\cal L}_{\sigma_k}(\xi,u; x^k)\Big\}
		\end{equation}
		to satisfy the conditions \eqref{stopA}, \eqref{stopB1}, and \eqref{stopB2} below.
		
		\STATE $u^{k+1}=(x^k/\sigma_k- A^T \xi^{k+1})-{\rm Prox}_{p}( x^k/\sigma_k- A^T\xi^{k+1})$.
		
		\STATE $x^{k+1}  =x^k-\sigma_k (A^T \xi^{k+1}+u^{k+1})=\sigma_k{\rm Prox}_{ p}({x^k}/{\sigma_k}- A^T \xi^{k+1}).$
		
		\STATE Update $\sigma_{k+1} \uparrow \sigma_\infty\leq \infty$,  $k\leftarrow k+1$, and go to Step 1.
	\end{algorithmic}
\end{algorithm}

For the {\sc Ssnal} algorithm, we use the following implementable stopping criteria as in \cite{rockafellar1976augmented,rockafellar1976monotone}:
\begin{align}
\| \nabla \psi_k (\xi^{k+1}) \|&\leq \epsilon_k/\sqrt{\sigma_k},\ \sum_{k=0}^{\infty} \epsilon_k <\infty,\tag{A}\label{stopA}\\
\| \nabla \psi_k (\xi^{k+1}) \|&\leq  \delta_k \sqrt{\sigma_k}\|A^T \xi^{k+1}+u^{k+1}\|,\ \sum_{k=0}^{\infty} \delta_k <\infty,\tag{B1}\label{stopB1}\\
\| \nabla \psi_k (\xi^{k+1}) \|&\leq \delta'_k \| A^T \xi^{k+1}+u^{k+1}\|,\ 0 \leq \delta'_k \rightarrow 0,\tag{B2}\label{stopB2}
\end{align}
where $\{\epsilon_k\}$, $\{\delta_k\}$, $\{\delta'_k\}$ are given nonnegative error tolerance sequences.

Define the following maximal monotone operators \cite{rockafellar1976augmented}
\[
{\cal T}_f(x):=\partial f(x),\ {\cal T}_l(\xi,u;x):=\{(\xi',u',x')\mid (\xi',u',-x')\in \partial l(\xi,u;x)\}.
\]
The piecewise linear-quadratic property of $f$ leads to the fact that ${\cal T}_f$ and ${\cal T}_l$ satisfy the error bound condition \cite{luque1984asymptotic} at point $0$ with positive modulus $a_{f}$ and $a_l$, respectively \cite{robinson1981some,sun1986monotropic}. That is to say, there exists $\varepsilon>0$ such that if ${\rm dist}(0,{\cal T}_f(x))\leq \varepsilon$, then
\begin{equation}\label{define_af}
{\rm dist}(x, \Omega_p)\leq a_f {\rm dist}(0,{\cal T}_f(x)).
\end{equation}
Besides, there exists $\varepsilon'>0$ such that if ${\rm dist}(0,{\cal T}_l(\xi,u;x))\leq \varepsilon'$, then
\begin{equation}\label{define_al}
{\rm dist}((\xi,u,x), (\xi^*,u^*)\times \Omega_p)\leq a_l {\rm dist}(0,{\cal T}_l(\xi,u;x)),
\end{equation}
where $(\xi^*,u^*)$ is the unique optimal solution of \eqref{D}.

The global and local convergence of the {\sc Ssnal} algorithm have been studied in \cite{rockafellar1976augmented,rockafellar1976monotone,luque1984asymptotic}. Here we simply state some relevant results.

\begin{theorem}\label{local_global}
	(1)   Let $\{(\xi^k,u^k,x^k)\}$ be the infinite sequence generated by the {\sc Ssnal} algorithm with stopping criterion \eqref{stopA}. Then, the sequence $\{x^k\}$ converges to an optimal solution of \eqref{P}. In addition, $\{(\xi^k,u^k)\}$  converges to the unique optimal solution $(\xi^*,u^*)$ of \eqref{D}.
	
	(2) For the sequence  $\{(\xi^k,u^k,x^k)\}$  generated by the {\sc Ssnal} algorithm with stopping criteria \eqref{stopA} and \eqref{stopB1},  one has that for all $k$ sufficiently large,
	\begin{equation}\label{convergence_x}
	{\rm dist}(x^{k+1},\Omega_p)\leq \theta_k{\rm dist}(x^{k},\Omega_p),
	\end{equation}
	where $\theta_k=(a_f(a_f^2+\sigma_k^2)^{-1/2}+2\delta_k)(1-\delta_k)^{-1}\rightarrow \theta_{\infty}=a_f(a_f^2+\sigma_{\infty}^2)^{-1/2}<1$ as $k\rightarrow +\infty$, and $a_f$ is from \eqref{define_af}. If the stopping criterion \eqref{stopB2} is also satisfied, it holds that for $k$ sufficiently large,
	\begin{equation}\label{convergence_xiu}
	\|(\xi^{k+1},u^{k+1})-(\xi^*,u^*)\|\leq \theta_k'\|x^{k+1}-x^k\|,
	\end{equation}
	where $\theta_k'=a_l(1+\delta_k')/\sigma_k\rightarrow a_l/\sigma_{\infty}$ as $k\rightarrow +\infty$, and $a_l$ is from \eqref{define_al}.
\end{theorem}
\begin{proof}
	The first part of this theorem can be obtained from \cite[Theorem 4]{rockafellar1976augmented}. Since ${\cal T}_f$ and ${\cal T}_l$ satisfy the error bound condition, it follows from \cite[Theorem 2.1]{luque1984asymptotic} that \eqref{convergence_x} holds. If \eqref{stopA}, \eqref{stopB1} and \eqref{stopB2} are all satisfied, combing \cite{cui2016asymptotic} with \cite[Remark 1]{li2016highly}, we get the desired result that \eqref{convergence_xiu} holds. This completes the proof.
\end{proof}

\subsection{A semismooth Newton method for the subproblem} \label{ssn_dual}

In this subsection, we present an efficient semismooth Newton method for solving the {\sc Alm} subproblem \eqref{d:alm-sub}.
Given $\tilde{x}\in \Re^n$ and $\sigma>0$, we consider the following minimization problem
\begin{equation}\label{eq:subprob}
\min_{\xi\in\Re^m}\left\{\psi(\xi):=\inf_{u}{\cal L}_{\sigma}( \xi,u; \tilde{x})\right\},
\end{equation}
where using the Moreau's identity, we get that
\begin{equation*}
\begin{split}
\psi(\xi)=\inf_{u}{\cal L}_{\sigma}( \xi,u; \tilde{x}) & = \frac{1}{2}\|\xi\|^2+\langle b,\xi\rangle +p^*({\rm Prox}_{p^*/\sigma}(-A^T \xi +\tilde{x}/\sigma) ) \\
& \quad+\frac{1}{2\sigma}\|{\rm Prox }_{\sigma p}( -\sigma A^T \xi +\tilde{x}) \|^2-\frac{1}{2\sigma}\|\tilde{x}\|^2.
\end{split}
\end{equation*}
Since $\psi(\cdot)$ is strongly convex and continuously differentiable, the minimization problem \eqref{eq:subprob} has a unique solution $\hat{\xi}$ which can be obtained via solving the following nonsmooth equation
\begin{equation}\label{d-eq-xi}
0=\nabla \psi(\xi)=\xi+b -A {\rm Prox }_{\sigma p}(\tilde{x} -\sigma A^T \xi)=\xi+b -\sigma A {\rm Prox }_{p}(\tilde{x}/\sigma-A^T \xi).
\end{equation}
Here we use the fact that ${\rm Prox}_{\sigma p}(z)=\sigma {\rm Prox}_p(z/\sigma)$ for any $z\in \Re^n$.

Define the multifunction ${\cal V}:\Re^m\rightrightarrows\Re^{m\times m}$ by:
\[
{\cal V}(\xi) :=\Big\{V \in \Re^{m\times m} \mid V= I_m+\sigma A M A^T ,M\in {\cal M}(\tilde{x}/\sigma - A^T \xi)\Big\},
\]
where ${\cal M}(\cdot)$ is the multifunction defined in \eqref{def_M}. By virtue of Theorem \ref{theo_M}, we know that ${\cal V}$ is nonempty, compact, and upper-semicontinuous. It is obvious that for any $\xi\in\Re^m$, all elements of ${\cal V}(\xi)$ are symmetric and positive definite. In addition, $\nabla \psi$ is $\gamma$-order semismooth on $\Re^m$ with respect to ${\cal V}$, for any $\gamma>0$.

We shall apply a semismooth Newton ({\sc Ssn}) method to solve (\ref{d-eq-xi}) as follows and could expect to get a fast superlinear or even quadratic convergence.

\begin{algorithm}
	\caption{\small {\bf: ({\sc Ssn}) A semismooth Newton method for solving \eqref{d-eq-xi} }}
	\hspace*{0.02in} \raggedright {\bf Input:} $\mu \in (0, 1/2)$, $\bar{\eta} \in (0, 1)$, $\tau \in (0,1]$, $\delta \in (0, 1)$, $\xi^0,\tilde{x},\sigma$, and $j=0$.\\
	
	\begin{algorithmic}[1]
		\STATE   Choose $V_j \in {\cal V}(\xi^j)$. Solve the following linear system
		\begin{equation}\label{d-eqn-epsk}
		V_j h=-\nabla \psi(\xi^j)
		\end{equation}
		exactly or by the conjugate gradient (CG) algorithm
		to find $h^j$ such that
		\[
		\| V_j h^j  + \nabla \psi(\xi^j)\|\leq\min(\bar{\eta}, \| \nabla \psi(\xi^j)\|^{1+\tau}).
		\]
		
		\STATE  Set $\alpha_j = \delta^{m_j}$, where $m_j$ is the first nonnegative integer $m$ for which
		\begin{equation*}
		\psi(\xi^j + \delta^{m} h^j)\leq \psi(\xi^j)+ \mu \delta^{m}\langle \nabla \psi(\xi^j), h^j \rangle.
		\end{equation*}
		
		\STATE Set $\xi^{j+1} = \xi^j + \alpha_j \, h^j$, $j\leftarrow j+1$, and go to Step 1.
	\end{algorithmic}
\end{algorithm}

The convergence analysis for the {\sc Ssn} algorithm can be established as in \cite[Theorem 3]{li2018efficiently}.
\begin{theorem}
	Let $\{\xi^j\}$ be the infinite sequence generated by the {\sc Ssn} algorithm. Then, $\{\xi^j\}$ converges to the unique optimal solution $\hat{\xi}$ of problem \eqref{eq:subprob} and $\|\xi ^{j+1}-\hat{\xi}\|=O(\|\xi ^{j}-\hat{\xi}\|^{1+\tau})$.
\end{theorem}
\begin{proof}
	According to \cite[Proposition 3.3 \& Theorem 3.4]{zhao2010newton} and the fact that $\psi(\cdot)$ is strongly convex, $\{\xi^j\}$ converges to the unique optimal solution $\hat{\xi}$ of problem \eqref{eq:subprob}. Since ${\cal V}(\cdot)$ is a nonempty, compact valued, and upper-semicontinuous set-mapping, and all elements of ${\cal V}(\hat{\xi})$ are nonsingular, it follows from \cite[Lemma 7.5.2]{facchinei2007finite} that $\{\|V_j^{-1}\|\}$ is uniformly bounded for sufficiently large $j$. In addition, $\nabla \psi$ is strongly semismooth on $\Re^m$ with respect to ${\cal V}$. By mimicking the proofs in \cite[Theorem 3.5]{zhao2010newton}, we know that
	there exists $\hat{\delta}>0$ such that for all sufficiently large $j$, one has
	\begin{equation}\label{condition_yj}
	\|\xi ^{j}+h^j-\hat{\xi}\|=O(\|\xi ^{j}-\hat{\xi}\|^{1+\tau}) 
	\end{equation}
	and  
	\[
	-\langle \nabla\psi(\xi^j),h^j\rangle \geq \hat{\delta}\|h^j\|^2.
	\]
	By using \eqref{condition_yj}, \cite[Proposition 7]{li2018efficiently} and \cite[Proposition 8.3.18]{facchinei2007finite}, we can derive that for $\mu\in(0,1/2)$, there exists an integer $j_0$ such that for all $j\geq j_0$,
	\[
	\psi(\xi^j+h^j)\leq \psi(\xi^j)+\mu\langle \nabla\psi(\xi^j),h^j\rangle,
	\]
	which implies that  $\xi^{j+1}=\xi^j+h^j$ for all $j\geq j_0$. Combining with \eqref{condition_yj}, we complete the proof.
\end{proof}

\subsection{On the implementation of the {\sc Ssnal} algorithm for the dual problem}

The most time consuming step in our algorithm is in solving the Newton equation \eqref{d-eqn-epsk}. In this subsection, we shall design an efficient procedure to solve it.

Given $y:= {\tilde{x}}/{\sigma}-A^T\xi$, we have already known that
\[
M=\Theta Q\in {\cal M}(y),
\]
where $Q=P_y^T (I_n-B^T(\Sigma B B^T \Sigma)^{\dagger}B)P_y $, and $\Sigma$, $\Theta$ are defined in \eqref{def_sigma}-\eqref{def_theta}, respectively. For the Newton equation \eqref{d-eqn-epsk}, we need to deal with the matrix $AMA^T$. Thus it is important to analyze its structure in order to solve \eqref{d-eqn-epsk} efficiently.

Note that $\Sigma={\rm Diag}\{ \Lambda_1,\cdots,\Lambda_N\}$ is an $N$-block diagonal matrix with each $\Lambda_i$ being either a zero matrix or an identity matrix, and any two consecutive blocks are not of the same type, we can apply Proposition \ref{prop_gamma} to simplify our computation. Let $J:=\{j\mid \Lambda_j=I_{n_j},j=1,\cdots,N\}$. Then we have
\begin{equation*}
Q=P_y^T(H+U_J U_J^T)P_y=P_y^T H P_y+P_y^T U_J  U_J^T P_y,
\end{equation*}
where the $N$-block diagonal matrix $H={\rm Diag}(\Upsilon_1,\cdots,\Upsilon_N)\in\Re^{n\times n}$ is defined by
\begin{equation*}
\Upsilon_i=\left\{
\begin{array}{ll}
\textbf{O}_{n_i+1}, & \mbox{if $i\in J$},\\[5pt]
I_{n_i}, & \mbox{if $i \notin J $ and $i\in\{1,N\}$},\\[5pt]
I_{n_i-1}, & \mbox{otherwise},
\end{array}\right.
\end{equation*}
and $U_J$ is defined in Proposition \ref{prop_gamma}.

Since $M=\Theta Q$ is symmetric, it holds that $\Theta Q=M =M^T= Q\Theta$. Due to the fact that $\Theta$ is a $0$-$1$ diagonal matrix, we have that $\Theta=\Theta^2$ and hence
\[
M=\Theta Q=\Theta(\Theta Q)=\Theta(Q \Theta).
\]
Thus, after plugging in the derived formula for $Q$, we get that
\[
M=\Theta \widetilde{H}\Theta+\Theta P_y^T U_J(P_y^T U_J)^T\Theta,
\]
where the matrix
\begin{equation*}
\widetilde{H}=P_y^TH P_y = {\rm Diag} (P_y^T{\rm diag}(H))
\end{equation*}
is also a $0$-$1$ diagonal matrix. It follows that
\begin{equation*}
AMA^T= A \Theta \widetilde{H} \Theta A^T +A\Theta P_y^T U_J(P_y^T U_J)^T\Theta A^T.
\end{equation*}
Define the following index sets
\begin{equation*}
\alpha:=\Big\{i\mid \theta_i=1,i\in \{1,\cdots,n\}\Big\},\ \gamma:=\Big\{i\mid \widetilde{h}_{i}=1,i\in \alpha \Big\},
\end{equation*}
where $\theta_i$ and $\widetilde{h}_i$ are the $i$-th diagonal entries of $\Theta$ and $\widetilde{H}$, respectively. Then, we immediately get the following formula
\begin{equation*}
A \Theta \widetilde{H} \Theta A^T=A_{\alpha }\widetilde{H}A_{\alpha }^T=A_{\gamma}A_{\gamma}^T,
\end{equation*}
where $A_{\alpha}\in \Re^{m\times|\alpha |}$ and $A_{\gamma}\in \Re^{m\times|\gamma |}$ are two sub-matrices obtained from $A$ by extracting those columns with indices in $\alpha$ and $\gamma$, respectively. Furthermore, we have that
\begin{equation*}
A\Theta P_y^T U_J(P_y^T U_J)^T\Theta A^T=A_{\alpha}P_y^T U_J(P_y^T U_J)^T  A_{\alpha}^T=A_{\alpha}\widetilde{U}\widetilde{U} ^T  A_{\alpha}^T,
\end{equation*}
where $\widetilde{U}\in \Re^{|\alpha|\times t}$ is a sub-matrix obtained from $\Theta(P_y^TU_J)$ by extracting those rows with indices in $\alpha$ and the zero columns in $\Theta(P_y^TU_J)$ being removed. Finally, we obtain that
\begin{equation*}
AMA^T=A_{\gamma}A_{\gamma}^T+A_{\alpha}\widetilde{U}\widetilde{U} ^T  A_{\alpha}^T.
\end{equation*}
Li et al. \cite{li2018efficiently} referred to the above structure of $AMA^T$ and that of $I_m+\sigma AMA^T$ inherited from $M$ as the second-order structured sparsity. They also gave a thorough analysis of computational cost, which is quite similar in our case. Without considering the cost of computing $P_y^T{\rm diag}(H)$ and $P_y^T U_J$, the arithmetic operations of computing $AMA^T$ and $AMA^Td$ for a given vector $d$ are $O(m|\alpha|(m+t))$ and  $O(|\alpha|(m+t))$, respectively. With the use of the Sherman-Morrison-Woodbury formula \cite{golub2012matrix}, the computational cost can be further reduced. We omit the details here.

\section{A semismooth Newton proximal augmented Lagrangian method for the primal problem}\label{sect:ssnalP}

The augmented Lagrangian method ({\sc Alm}) for the dual problem \eqref{D} is expected to be efficient for the case when $m\ll n$, since the semismooth Newton system \eqref{d-eqn-epsk} is of dimension $m$ by $ m$. But for the case when $m \gg n$, as we shall see later in the numerical experiments, it is naturally more efficient to apply the {\sc Alm} on  the primal problem to avoid having to deal with a large $m$ by $m$ linear system in each semismooth Newton iteration. In this section, we will derive a semismooth Newton proximal {\sc Alm} for the primal problem.

First we rewrite the primal problem as
\begin{equation}\tag{P'}\label{P'}
\min_{x\in \Re^n, z\in \Re^n} \ \displaystyle\Big\{\frac{1}{2} \|Ax-b\|^2+  p(z) \mid  x-z=0\Big\}.
\end{equation}
The dual of \eqref{P'} is given as
\begin{equation}\tag{D'}\label{D'}
\max_{y\in \Re^n,v\in \Re^n} \ \displaystyle\Big\{-\frac{1}{2}\| Av-b\|^2-\langle b,Av-b\rangle -p^*(-y)  \mid  A^T(Av-b)-y=0\Big\}.
\end{equation}
Given $\sigma>0$, the augmented Lagrangian function of problem \eqref{P'} is given by
\begin{equation*}
\widetilde{{\cal L}}_\sigma( x,z; y):=\frac{1}{2} \|Ax-b\|^2+  p(z) -\langle y,x-z\rangle +\frac{\sigma}{2}\|x-z\|^2.
\end{equation*}

\subsection{A semismooth Newton proximal augmented Lagrangian method for \eqref{P'}}

The semismooth Newton proximal {\sc Alm} for \eqref{P'} has a similar framework as the {\sc Ssnal} algorithm for \eqref{D}. For simplicity,
we just state Algorithm \ref{pssnal} here without giving the detailed derivation.

\begin{algorithm}
	\caption{\small {\bf: (p-{\sc Ssnal}) A semismooth Newton augmented Lagrangian method for \eqref{P'}} } \label{pssnal}
	\hspace*{0.02in} \raggedright {\bf Input:} $\sigma_0 >0$, $( x^0,z^0,y^0)\in\Re^n\times \Re^n \times \Re^n$, and $k=0$.\\
	
	\begin{algorithmic}[1]
		\STATE  Adapt the semismooth Newton method to approximately compute
		\begin{equation} \label{p:alm-sub}
		x^{k+1}\approx\underset{x\in\Re^n}{\rm argmin} \Big\{ \phi_k (x):=\widetilde{{\cal L}}_{\sigma_k}( x,{\rm Prox}_{p/ \sigma_k}(x-y_k/\sigma_k); y_k)+\frac{1}{2\sigma_k}\|x-x^k\|^2\Big\}
		\end{equation}
		to satisfy the condition \eqref{stopA2} below.
		\STATE $z^{k+1}={\rm Prox}_{p/ \sigma_k}(x^{k+1}-y_k/\sigma_k).$
		\STATE
		$y^{k+1}=y^k -\sigma_k (x^{k+1}-z^{k+1}).$
		\STATE Update $\sigma_{k+1} \uparrow \sigma_\infty\leq \infty$,  $k\leftarrow k+1$, and go to Step 1.
	\end{algorithmic}
\end{algorithm}

In the p-{\sc Ssnal} algorithm, we apply a semismooth Newton method ({\sc Ssn}) to solve \eqref{p:alm-sub} with the following stopping criterion:
\begin{equation}\tag{A2}\label{stopA2}
\|\nabla \phi_k (x^{k+1})\|\leq \frac{\epsilon_k}{\sigma_k}\min(1,\|(x^{k+1},z^{k+1},y^{k+1})-(x^k,z^k,y^k)\|),\ \sum_{k=0}^{\infty} \epsilon_k <\infty.
\end{equation}

The proximal {\sc Alm} has been studied in \cite[Section 5]{rockafellar1976augmented}, which is also called the proximal method of multipliers. We can get the global convergence and local linear convergence of the proximal {\sc Alm} without any difficulty from \cite{rockafellar1976augmented,rockafellar1976monotone,luque1984asymptotic}.

\subsection{A semismooth Newton method for solving \eqref{p:alm-sub}}

Similar to the case of the {\sc Ssnal} algorithm, the most expensive step in each iteration of the p-{\sc Ssnpal} algorithm is in solving the subproblem \eqref{p:alm-sub}. Given $\sigma>0$ and $(\tilde{x},\tilde{y})\in \Re^n\times\Re^n$, we adapt a semismooth Newton method to solve a typical subproblem of the following form
\begin{equation*}
\min_{x\in \Re^n} \phi(x):=\widetilde{{\cal L}}_{\sigma}( x,{\rm Prox}_{p/ \sigma}(x-\tilde{y}/\sigma);\tilde{y})+\frac{1}{2\sigma}\|x-\tilde{x}\|^2.
\end{equation*}
Since $\phi(\cdot)$ is continuously differentiable and strongly convex, the above optimization problem has a unique solution $\hat{x}$. Thus, it is equivalent to solving the following nonsmooth equation
\begin{equation}\label{p-eq-x}
\begin{split}
0=\nabla \phi(x) &= A^T(Ax-b)+\sigma x -\tilde{y}-\sigma{\rm Prox}_{p/ \sigma}(x-\tilde{y}/\sigma)+(x-\tilde{x})/\sigma \\
&=A^T(Ax-b)+(\sigma+1/\sigma) x -(\tilde{y}+\tilde{x}/\sigma)-{\rm Prox}_{p}(\sigma x-\tilde{y}).
\end{split}
\end{equation}

Define the multifunction ${\cal U}:\Re^n\rightrightarrows\Re^{n\times n}$ by
\[
{\cal U}(x):=\Big\{U \in \Re^{n\times n} \mid U= A^T A +\sigma(I_n- M) +\frac{1}{\sigma}I_n,M\in {\cal M}(\sigma x-\tilde{y})\Big\},
\]
where ${\cal M}(\cdot)$ is defined as in \eqref{def_M}. From Theorem \ref{theo_M}, we obtain that ${\cal U}$ is a nonempty, compact valued and upper-semicontinuous multifunction with its elements being symmetric and positive definite. Besides, $\nabla \phi$ is $\gamma$-order semismooth on $\Re^n$ with respect to ${\cal U}$ for all $\gamma>0$. Thus we can apply a semismooth Newton ({\sc Ssn}) method to solve (\ref{p-eq-x}). Similar to the results in Section \ref{ssn_dual}, the {\sc Ssn} method has a fast superlinear or even quadratic convergence.

The efficiency of the  {\sc Ssn} method depends on the generalized Jacobian of $\nabla \phi(\hat{x})$. Next, we characterize the positive definiteness of the elements in ${\cal U}(\hat{x})$ in the following proposition.

\begin{proposition}\label{assum}
	For any
	$
	U=A^T A+\sigma(I_n-M)+\frac{1}{\sigma}I_n\in{\cal U}(\hat{x}),
	$
	we have that
	\[
	\lambda_{\min}(U)\geq\lambda_{\min}(A^T A+\sigma(I_n-M))+\frac{1}{\sigma}\geq\lambda_{\min}(A^T A)+\sigma\lambda_{\min}(I_n-M)+\frac{1}{\sigma}\geq\frac{1}{\sigma}.
	\]
\end{proposition}
\begin{proof}
	From Theorem \ref{theo_M}, we know that for any $M\in {\cal M}(\sigma \hat{x}-\tilde{y})$, $I_n-M$ is symmetric and positive semidefinite, which yields that the desired conclusion holds trivially.
\end{proof}

\begin{remark}
	When the columns of $A$ are linearly independent, for any $U\in {\cal U}(\hat{x})$, we have that
	\[
	\lambda_{\min}(U)\geq\lambda_{\min}(A^T A)+\sigma\lambda_{\min}(I_n-M)+\frac{1}{\sigma}\geq\lambda_{\min}(A^T A)+\frac{1}{\sigma}.
	\]
	In that case, $U$ is positive definite if we do not add the proximal term $\frac{1}{2\sigma_k}\|x-x^k\|^2$ in \eqref{p:alm-sub}.
	Since here we mainly focus on the case when $m\gg n$, the columns of $A$ are very likely to be linearly independent.
\end{remark}

\section{Numerical experiments}\label{sect:NumRes}

In this section, we will evaluate the performance of our {\sc Ssnal} algorithm for solving the clustered lasso problems on the high-dimension-low-sample setting and the high-sample-low-dimension setting, respectively. For simplicity, we use the following abbreviations. {\sc Ssnal} represents the semismooth Newton augmented Lagrangian method, {\sc Admm} represents the alternating direction method of multipliers, {\sc iAdmm} represents the inexact {\sc Admm}, {\sc LAdmm} represents the linearized {\sc Admm} and {\sc Apg} represents the accelerated proximal gradient method. We implemented {\sc Admm}, {\sc iAdmm} and {\sc LAdmm} in MATLAB with the step-length set to be $1.618$.

In our experiments, the regularization parameters $\beta$ and $\rho$ in the clustered lasso problem \eqref{given-P} are chosen to have the form
\[
\beta=\alpha_1 \|A^T b\|_{\infty},\ \rho=\alpha_2 \beta,
\]
where $0<\alpha_1<1$ and $\alpha_2>0$. To produce reasonable clustering results, we choose $\alpha_2 =O( 1/n)$ to make sure that the two penalty terms have the same magnitude of influence.

We stop the tested algorithms according to some specified stopping criteria, which will be given in the following subsections. Besides, the algorithms will be stopped when they reach the maximum computation time of $3$ hours or the pre-set maximum number of iterations ($100$ for {\sc Ssnal}, and $20000$ for {\sc Admm}, {\sc iAdmm}, {\sc LAdmm}, {\sc Apg}). All our computational results are obtained by running MATLAB (version 9.0) on a windows workstation (12-core, Intel Xeon E5-2680 @ 2.50GHz, 128 G RAM).

\subsection{First order methods}\label{sect:FOMs}

For comparison purpose, we summarize two types of first-order methods that are suitable for solving the clustered lasso problem. An important point to mention here is that the proximal mapping given in Section \ref{subsection_prox_mapping} plays a crucial role in the projection steps of these methods. Indeed, the new characteristic of the clustered lasso regularizer vastly improves the performance of the first-order methods as the computation of the proximal mapping is now much cheaper.

\paragraph{Alternating direction method of multipliers for \eqref{D}}
We start by adapting the widely-used alternating direction method of multipliers ({\sc Admm}) \cite{eckstein1992douglas,gabay1976dual,glowinski1975approximation} for solving \eqref{D}, which can be described as Algorithm \ref{dadmm}.

\begin{algorithm}
	\caption{\small {\bf : (d-{\sc Admm}) An alternating direction method of multipliers for \eqref{D}}}\label{dadmm}
	\hspace*{0.02in} \raggedright {\bf Input:} $\kappa \in (0, (1+ \sqrt{5})/2)$, $\sigma >0$, $x^0\in\Re^{n}$, $u^0 \in \Re^n$, and $k=0$.\\
	
	\begin{algorithmic}[1]
		\STATE  Compute
		\begin{equation}\label{ADMM_xi}
		\xi^{k+1}\approx \underset{\xi\in\Re^m}{\mbox{argmin}}\,{\cal L}_{\sigma}(\xi,u^{k};x^k).
		\end{equation}
		\STATE $
		u^{k+1}= \underset{u\in\Re^{n}}{\mbox{argmin}}\, {\cal L}_{\sigma}( \xi^{k+1},u;x^{k})={\rm Prox}_{p^*/\sigma}(-A^T \xi^{k+1}+ {x^k}/{\sigma}).
		$
		\STATE $x^{k+1} = x^k-\kappa \sigma (A^T \xi^{k+1}+u^{k+1})$.
		\STATE  $k\leftarrow k+1$, and go to Step 1.
	\end{algorithmic}
\end{algorithm}

Note that in practice, $\kappa$ should be chosen to be at least $1$ for faster convergence. For the subproblem \eqref{ADMM_xi}, the optimality condition that $\xi^{k+1}$ must satisfy is given by
\begin{equation*}
(I_m+\sigma A A^T)\xi =-b+ A(x^k-\sigma u^{k}).
\end{equation*}
The linear system of equation of the form $(I_m+\sigma A A^T)\xi = h$ has to be solved repeatedly with a different right-hand side vector $h$. One can solve this linear system directly or use an iterative solver such as the preconditioned conjugate gradient ({\sc Pcg}) method.

The convergence results of the classical {\sc Admm} with the subproblems solved exactly have been discussed in \cite{fazel2013hankel}, while the convergence analysis of the inexact {\sc Admm} can be found in \cite{chen2017efficient}.
The linearized {\sc Admm} algorithm \cite{zhang2011unified} can also be used to solve this problem by linearizing the quadratic term in \eqref{ADMM_xi}. It is worthwhile to mention that inexact {\sc Admm} and linearized {\sc Admm} are often used in the case when $m$ is large.

\paragraph{Alternating direction method of multipliers for \eqref{P'}}
Next we present the {\sc Admm} algorithm for \eqref{P'}, which is described as Algorithm \ref{pADMM}.
\begin{algorithm}
	\caption{\small {\bf : (p-{\sc Admm}) An alternating direction method of multipliers for \eqref{P'} }}\label{pADMM}
	\hspace*{0.02in} \raggedright {\bf Input:} $\kappa \in (0, (1+ \sqrt{5})/2)$, $\sigma >0$, $z^0\in\Re^{n}$, $y^0 \in \Re^n$, and $k=0$.\\
	
	\begin{algorithmic}[1]
		\STATE  Compute
		\begin{equation}\label{pADMM_x}
		x^{k+1} \approx \underset{x\in \Re^n}{\mbox{argmin}}\, \widetilde{{\cal L}}_{\sigma}( x,z^{k}; y^k).
		\end{equation}
		\STATE $z^{k+1} = \underset{z\in\Re^{n}}{\mbox{argmin}} \,\widetilde{{\cal L}}_{\sigma}( x^{k+1},z; y^k) ={\rm Prox}_{p/ \sigma}(x^{k+1}-y_k/\sigma)$.
		\STATE $y^{k+1} = y^k-\kappa \sigma (x^{k+1}-z^{k+1})$.
		\STATE  $k\leftarrow k+1$, and go to Step 1.
	\end{algorithmic}
\end{algorithm}

Note that, for the subproblem \eqref{pADMM_x}, $x^{k+1}$ is the solution of the following linear system
\begin{equation*}
(\sigma I_n +A^T A )x = A^Tb+\sigma(z^{k}+ {y_k}/{\sigma}).
\end{equation*}
Direct solvers and iterative solvers both can be used here.

\paragraph{An accelerated proximal gradient method of \eqref{P}}
Since the function $\|Ax-b\|^2/2$ in \eqref{P} has Lipschitz continuous gradient (with Lipschitz constant $L$, which is the largest eigenvalue of $A^T A$), one can attempt to use the accelerated proximal gradient ({\sc Apg}) method in \cite{beck2009fast} to solve \eqref{P}. The basic template of the {\sc Apg} algorithm is given in Algorithm \ref{algo-apg} below.

\begin{algorithm}
	\caption{\small {\bf : ({\sc Apg}) An accelerated proximal gradient method for \eqref{P} }}\label{algo-apg}
	\hspace*{0.02in} \raggedright {\bf Input:} $\varepsilon>0$, $w^0=x^0\in\Re^n$, $t_0=1$, and $k=0$.\\
	
	\begin{algorithmic}[1]
		\STATE  Compute
		\begin{equation*}
		x^{k+1}
		={\rm Prox}_{p/L}(w^k-L^{-1}A^T(Aw^k-b) ).
		\end{equation*}
		\STATE Set $t_{k+1}= (1+\sqrt{1+4 t_k^2}\,)/2.$
		\STATE Update $w^{k+1}=x^{k+1}+(t_k-1)/t_{k+1}(x^{k+1}-x^k).$
		\STATE $k\leftarrow k+1$, and go to Step 1.
	\end{algorithmic}
\end{algorithm}

It is clear that the practical performance of the {\sc Apg} algorithm hinges crucially on whether one can compute the proximal mapping ${\rm Prox}_{\nu p}(y)$ for any $y\in \Re^n$ and $\nu>0$ efficiently. Fortunately, we have provided an analytical solution to this problem in Section \ref{subsection_prox_mapping}.

\subsection{Stopping criteria}
Since the primal problem \eqref{P} is unconstrained, it is reasonable to measure the accuracy of an approximate optimal solution $(\xi,u,x)$ for problem \eqref{D} and problem \eqref{P} by the relative duality gap and dual infeasibility. Specifically, let
\begin{equation*}
\mbox{pobj}:=\frac{1}{2}\|Ax-b\|^2+p(x),\quad
\mbox{dobj}:=-\frac{1}{2}\|\xi\|^2- \langle b,\xi \rangle
\end{equation*}
be the primal and dual objective function values. The relative duality gap and the relative dual infeasibility are given as
\begin{equation*}
\eta_{gap}:=\frac{|\mbox{pobj}-\mbox{dobj}|}{1+|\mbox{pobj}|+|\mbox{dobj}|},\quad
\eta_{D}:=\frac{\|A^T\xi+u\|}{1+\|u\|}.
\end{equation*}
Besides, the relative KKT residual of the primal problem \eqref{P}
\begin{equation}\label{define_eta}
\eta_{kkt} = \frac{\|x-{\rm Prox}_p(x-A^T(Ax-b))\|}{1+\|x\|+\|A^T(Ax-b)\|}
\end{equation}
can be adopted to measure the accuracy of an approximate optimal solution $x$.

\subsection{Numerical results for UCI datasets}
In this subsection, we conduct some experiments on the same large-scale UCI datasets $(A,b)$ as in \cite{li2018efficiently} that are originally obtained from the LIBSVM datasets \cite{chang2011libsvm}. All instances are in the high-dimension-low-sample setting. According to what we have discussed in Section \ref{sect:ssnal}, the dual approaches are better choices since we have $m\ll n$ in this setting.

For given tolerance $\epsilon$, we will terminate the {\sc Ssnal} algorithm when
\begin{equation}\label{maxeta}
\max\{\eta_{gap},\eta_{D},\eta_{kkt}\}\leq\epsilon.
\end{equation}

Table \ref{table_UCI_ssnal} gives the numerical results for {\sc Ssnal} when solving the clustered lasso problem \eqref{given-P} on UCI datasets. In the table, $m$ and $n$ denotes the number of samples and features, respectively. We use ${\rm nnz}(x)$ to denote the number of nonzeros in the solution $x$ using the following estimation
\[
{\rm nnz}(x):=\min\{k |\sum_{i=1}^k |\hat{x}_i|\geq 0.99999 \|x\|_1\},
\]
where $\hat{x}$ is obtained by sorting $x$ such that $|\hat{x}_1|\geq |\hat{x}_2| \geq \cdots \geq |\hat{x}_n|$. We also use ${\rm gnnz}(x)$ to denote the number of groups in the solution, where the pairwise ratios among the sorted elements in each group are between $5/6$ and $6/5$. In order to get reasonable grouping results, we regard the elements with absolute value below $10^{-4}$ to be in the same group.

In order to get a reasonable number of non-zero elements in the optimal solution $x$, we choose $\alpha_1 \in \{10^{-6},10^{-7}\}$ for the problems {\rm E2006.train} and {\rm E2006.test}, $\alpha_1 \in \{10^{-2},10^{-3}\}$ for problem {\rm triazines4}, $\alpha_1 \in \{10^{-5},10^{-6}\}$ for problem {\rm bodyfat} and $\alpha_1 \in \{10^{-3},10^{-4}\}$ for the other instances. As we mention before, when $\alpha_2 =O( 1/n)$, we can get reasonable clustering results. In total, we tested $54$ instances.

From Table \ref{table_UCI_ssnal}, we see that the {\sc Ssnal} algorithm is efficient and robust against different parameter selections. It can be observed that all the $54$ tested instances are successfully solved by {\sc Ssnal} in about $5$ minutes. In fact, for most of the cases, they are solved in less than one minute.

\begin{center}\scriptsize
	\renewcommand{\multirowsetup}{\centering}
	\renewcommand\arraystretch{1.2}
	\begin{longtable}{|c|c|c|c|c|c|c|}
		\caption{The performance of the {\sc Ssnal} algorithm on UCI datasets with different parameter selections. We terminate {\sc Ssnal} when $\max\{\eta_{gap},\eta_{D},\eta_{kkt}\}
			\leq 10^{-6}$. ${\rm nnz}(x)$ and ${\rm gnnz}(x)$ are obtained by {\sc Ssnal}. Time is shown in the format of (hours:minutes:seconds).} \label{table_UCI_ssnal}\\
		\hline
		proname   ($m$; $n$)  &  $\alpha_1$; $\alpha_2$   & ${\rm nnz}(x)$; ${\rm gnnz}(x)$ & $\mbox{pobj}$ & $\eta_{kkt}$ & $\max\{\eta_{gap},\eta_{D}\}$ & time\\
		\hline
		\endfirsthead
		
		\multicolumn{7}{c}{{\bfseries \tablename\ \thetable{} -- continued from previous page}} \\
		
		\hline
		proname   ($m$; $n$)  &  $\alpha_1$; $\alpha_2$   & ${\rm nnz}(x)$; ${\rm gnnz}(x)$ & $\mbox{pobj}$ & $\eta_{kkt}$ & $\max\{\eta_{gap},\eta_{D}\}$ & time\\
		\hline
		\endhead
		
		\hline
		\multicolumn{7}{|r|}{{Continued on next page}} \\
		\hline
		\endfoot
		
		\hline
		\endlastfoot
		
		\multirow{ 6}*{\tabincell{c}{E2006.train  (16087; 150360)\\ $\lambda_{\max}(AA^T)=$1.91e+05} } &1e-6; 1e-5 &  4;  4 & 1.19083+3 & 2.9-7 & 8.7-7& 05\\ &1e-6; 1e-6 & 22;  8 & 1.18031+3 & 8.2-9 & 7.2-8& 07\\ &1e-6; 1e-7 & 27;  6 & 1.17744+3 & 6.3-8 & 7.6-7& 06\\\cline{2-7} &1e-7; 1e-4 &  8;  5 & 1.18600+3 & 1.2-8 & 5.3-8& 06\\ &1e-7; 5e-5 & 36;  6 & 1.17237+3 & 4.4-8 & 4.7-8& 09\\ &1e-7; 1e-5 & 380;  6 & 1.10710+3 & 1.4-8 & 4.8-7& 01:20\\\hline \multirow{ 6}*{\tabincell{c}{E2006.test  (3308; 150358)\\ $\lambda_{\max}(AA^T)=$4.79e+04} } &1e-6; 1e-5 & 10;  5 & 2.38906+2 & 2.9-8 & 4.5-7& 04\\ &1e-6; 1e-6 & 35;  5 & 2.29669+2 & 2.7-9 & 9.9-8& 04\\ &1e-6; 1e-7 & 53;  5 & 2.27308+2 & 2.3-9 & 1.2-7& 04\\\cline{2-7} &1e-7; 1e-4 & 20;  7 & 2.34499+2 & 7.8-10 & 3.4-8& 04\\ &1e-7; 5e-5 & 76;  8 & 2.23445+2 & 9.0-9 & 2.3-7& 10\\ &1e-7; 1e-5 & 550;  5 & 1.74748+2 & 2.7-10 & 4.5-8& 04:00\\\hline \multirow{ 6}*{\tabincell{c}{log1p.train  (16087; 4272227)\\ $\lambda_{\max}(AA^T)=$5.86e+07} } &1e-3; 1e-6 &  3;  3 & 2.80871+3 & 7.8-8 & 7.8-8& 49\\ &1e-3; 1e-7 &  3;  3 & 1.58340+3 & 2.0-7 & 2.0-7& 55\\ &1e-3; 1e-8 &  5;  5 & 1.45745+3 & 8.7-8 & 8.7-8& 01:09\\\cline{2-7} &1e-4; 1e-6 & 38; 11 & 1.27870+3 & 3.9-7 & 4.0-7& 01:27\\ &1e-4; 5e-7 & 92;  5 & 1.18724+3 & 1.1-7 & 1.1-7& 02:30\\ &1e-4; 1e-7 & 321;  5 & 1.08486+3 & 1.8-7 & 1.8-7& 05:08\\\hline \multirow{ 6}*{\tabincell{c}{log1p.test  (3308; 4272226)\\ $\lambda_{\max}(AA^T)=$1.46e+07} } &1e-3; 1e-6 &  3;  2 & 6.27631+2 & 4.3-7 & 4.3-7& 40\\ &1e-3; 1e-7 &  4;  4 & 3.41971+2 & 4.0-8 & 4.0-8& 01:08\\ &1e-3; 1e-8 &  8;  5 & 3.10745+2 & 5.5-8 & 5.5-8& 01:06\\\cline{2-7} &1e-4; 1e-6 & 50;  6 & 2.61434+2 & 2.4-7 & 2.4-7& 01:47\\ &1e-4; 5e-7 & 172;  5 & 2.34526+2 & 2.0-7 & 2.0-7& 03:28\\ &1e-4; 1e-7 & 726;  5 & 1.90645+2 & 1.5-7 & 1.5-7& 04:26\\\hline \multirow{ 6}*{\tabincell{c}{pyrim5  (74; 201376)\\ $\lambda_{\max}(AA^T)=$1.22e+06} } &1e-3; 5e-5 & 48;  5 & 6.08424-1 & 2.1-7 & 2.2-7& 36\\ &1e-3; 1e-5 & 65;  7 & 1.94647-1 & 4.1-7 & 4.4-7& 28\\ &1e-3; 1e-6 & 91; 10 & 8.80020-2 & 2.2-7 & 2.3-7& 21\\\cline{2-7} &1e-4; 5e-5 & 102;  5 & 8.19911-2 & 6.2-7 & 6.6-7& 52\\ &1e-4; 1e-5 & 88;  5 & 2.67953-2 & 2.1-7 & 2.3-7& 01:07\\ &1e-4; 1e-6 & 83;  6 & 1.24353-2 & 4.7-7 & 6.2-7& 28\\\hline \multirow{ 6}*{\tabincell{c}{triazines4  (186; 635376)\\ $\lambda_{\max}(AA^T)=$2.07e+07} } &1e-2; 1e-5 & 341;  3 & 7.81486+0 & 8.2-8 & 2.6-7& 37\\ &1e-2; 1e-6 & 373;  6 & 3.00214+0 & 2.0-7 & 2.1-7& 56\\ &1e-2; 1e-7 & 411;  8 & 2.31560+0 & 8.9-8 & 8.9-8& 50\\\cline{2-7} &1e-3; 1e-5 & 611;  6 & 1.86205+0 & 1.7-7 & 1.7-7& 02:06\\ &1e-3; 1e-6 & 641;  6 & 7.53899-1 & 9.6-7 & 9.6-7& 02:00\\ &1e-3; 5e-7 & 877;  7 & 6.55800-1 & 5.6-7 & 5.6-7& 04:52\\\hline \multirow{ 6}*{\tabincell{c}{abalone  (4177; 6435)\\ $\lambda_{\max}(AA^T)=$5.21e+05} } &1e-3; 1e-4 & 25;  8 & 1.24134+4 & 4.7-7 & 4.8-7& 01\\ &1e-3; 5e-5 & 24;  8 & 1.19308+4 & 5.7-7 & 5.8-7& 01\\ &1e-3; 1e-5 & 26;  9 & 1.15154+4 & 4.8-7 & 4.8-7& 01\\\cline{2-7} &1e-4; 1e-4 & 50;  8 & 9.54332+3 & 1.4-7 & 1.6-7& 04\\ &1e-4; 5e-5 & 51;  5 & 9.42227+3 & 1.9-7 & 2.2-7& 04\\ &1e-4; 1e-5 & 62;  7 & 9.31717+3 & 5.1-7 & 6.1-7& 05\\\hline \multirow{ 6}*{\tabincell{c}{bodyfat  (252; 116280)\\ $\lambda_{\max}(AA^T)=$5.29e+04} } &1e-5; 5e-5 & 10;  5 & 2.09723-2 & 2.0-8 & 2.8-8& 04\\ &1e-5; 1e-5 & 20;  8 & 7.14784-3 & 3.7-7 & 6.4-7& 05\\ &1e-5; 1e-6 & 27;  6 & 3.93005-3 & 1.3-7 & 2.3-7& 06\\\cline{2-7} &1e-6; 5e-5 & 38;  6 & 2.55045-3 & 7.4-8 & 1.6-7& 07\\ &1e-6; 1e-5 & 78;  6 & 9.90203-4 & 6.5-8 & 1.3-7& 13\\ &1e-6; 1e-6 & 108;  7 & 5.93863-4 & 6.6-8 & 1.3-7& 11\\\hline \multirow{ 6}*{\tabincell{c}{housing  (506; 77520)\\ $\lambda_{\max}(AA^T)=$3.28e+05} } &1e-3; 5e-5 & 106;  9 & 6.69490+3 & 3.5-7 & 4.6-7& 07\\ &1e-3; 1e-5 & 139;  6 & 3.76003+3 & 3.7-8 & 3.9-8& 09\\ &1e-3; 1e-6 & 158;  5 & 2.88365+3 & 5.3-8 & 5.4-8& 08\\\cline{2-7} &1e-4; 5e-5 & 207;  6 & 1.94260+3 & 1.8-7 & 1.9-7& 42\\ &1e-4; 1e-5 & 255; 11 & 1.21114+3 & 4.2-7 & 5.8-7& 28\\ &1e-4; 1e-6 & 292;  9 & 9.54315+2 & 8.1-8 & 1.0-7& 22\\\hline 
	\end{longtable}
\end{center}

For comparison, we also conduct numerical experiments on {\sc Admm}, {\sc iAdmm}, {\sc LAdmm} and {\sc Apg}. We select two pairs of parameters for each dataset when computing. Let $\mbox{pobj}_{\mbox{\sc Ssnal}}$ be the optimal primal objective value obtained by {\sc Ssnal} with stopping criterion \eqref{maxeta}. Since the minimization problem \eqref{P} is unconstrained, it is reasonable to terminate a first-order algorithm when
\begin{equation}\label{definerelgap}
\eta_{rel}:=\frac{\mbox{pobj}-\mbox{pobj}_{\mbox{\sc Ssnal}}}{1+|\mbox{pobj}_{\mbox{\sc Ssnal}}|}\leq \epsilon_2,
\end{equation}
where $\mbox{pobj}$ is the primal objective value obtained by the first-order algorithm and $\epsilon_2$ is a given tolerance. Here, we treat $\mbox{pobj}_{\mbox{\sc Ssnal}}$ as an accurate approximate optimal objective value to \eqref{P} and stop the other algorithms by using the relative difference between the obtained primal objective value and $\mbox{pobj}_{\mbox{\sc Ssnal}}$.

Table \ref{table_UCI_com4} and Table \ref{table_UCI_com6} show the numerical results. In the tables, $t_{\mbox{\sc Ssnal}}$ represents the time needed by {\sc Ssnal} when using the stopping criterion \eqref{maxeta} with $\epsilon=10^{-6}$. We test two different choices of $\epsilon_2$. The results for $\epsilon_2=10^{-4}$ are shown in Table \ref{table_UCI_com4} and the results for $\epsilon_2=10^{-6}$ are shown in Table \ref{table_UCI_com6}.

When $\epsilon_2=10^{-4}$, we can see from Table \ref{table_UCI_com4} that {\sc Admm} is able to solve $18$ instances and {\sc iAdmm} can solve $17$ instances successfully. While {\sc LAdmm} and {\sc Apg} can solve $16$ and $12$ instances successfully, respectively. When $\epsilon_2=10^{-6}$, we can see from Table \ref{table_UCI_com6} that {\sc Admm} is able to solve $16$ instances and {\sc iAdmm} can solve $15$ instances. While for {\sc LAdmm} and {\sc Apg}, they can only solve $11$ and $3$ instances, respectively. We note that {\sc iAdmm} and {\sc LAdmm} are computationally more advantageous than {\sc Admm} when solving instances with large $m$, thus it is not surprising that {\sc Admm} is more efficient than {\sc iAdmm} and {\sc LAdmm} in solving the tested instances for which $m$ is not too large.

By comparing the computation time between {\sc Ssnal} and the first-order algorithms, we can see that {\sc Ssnal} takes much less time than the first-order algorithms but get much better results in almost all cases. If we require a high accuracy, then the first-order methods will take much longer time than {\sc Ssnal} and may not even achieve the required accuracy.

\begin{center}\scriptsize
	\renewcommand{\multirowsetup}{\centering}
	\renewcommand\arraystretch{1.3}
	\begin{longtable}{|c|c|c|c|c|} 	
		\caption{The performance of various algorithms on UCI datasets. In the table, "b" = {\sc Admm}, "c" ={\sc iAdmm}, "d" = {\sc LAdmm}, "e" = {\sc Apg}. We terminate the first-order algorithms when $\eta_{rel}\leq 10^{-4}$. $t_{\mbox{\sc Ssnal}}$ represents the time needed by {\sc Ssnal} when using stopping criterion $\max\{\eta_{gap},\eta_{D},\eta_{kkt}\}\leq 10^{-6} $. Time is shown in the format of (hours:minutes:seconds). } \label{table_UCI_com4}\\
		\hline
		& &  &  $\eta_{rel}$ & time\\
		\hline
		proname   & $\alpha_1$; $\alpha_2$     &  $t_{\mbox{\sc Ssnal}}$  & b $|$ c $|$ d $|$ e & b $|$ c $|$ d $|$ e \\
		\hline
		\endfirsthead
		
		\multicolumn{5}{c}{{\bfseries \tablename\ \thetable{} -- continued from previous page}} \\
		
		\hline
		& &  & $\eta_{rel}$ & time\\
		\hline
		proname   & $\alpha_1$; $\alpha_2$     &  $t_{\mbox{\sc Ssnal}}$  & b $|$ c $|$ d $|$ e & b $|$ c $|$ d $|$ e \\
		\hline
		\endhead
		
		\hline
		\multicolumn{5}{|r|}{{Continued on next page}} \\
		\hline
		\endfoot
		
		\hline
		\endlastfoot
		
		\multirow{ 2}*{\tabincell{c}{E2006.train} } &1e-6; 1e-7& 06 & 1.0-4 $|$ 1.0-4 $|$ 9.9-5 $|$ 9.6-5& 34:05 $|$ 06:47 $|$ 32 $|$ 53:30\\ &1e-7; 1e-5& 01:20 & 1.0-4 $|$ 9.9-5 $|$ 1.0-4 $|$ \red{5.5-3}& 41:51 $|$ 09:54 $|$ 01:19 $|$ 59:38\\\hline \multirow{ 2}*{\tabincell{c}{E2006.test} } &1e-6; 1e-7& 04 & 9.8-5 $|$ 9.8-5 $|$ 9.9-5 $|$ \red{3.9-3}& 03:41 $|$ 03:09 $|$ 32 $|$ 23:57\\ &1e-7; 1e-5& 04:00 & 1.0-4 $|$ 9.9-5 $|$ 1.0-4 $|$ \red{6.6-2}& 04:41 $|$ 04:28 $|$ 01:27 $|$ 24:21\\\hline \multirow{ 2}*{\tabincell{c}{log1p.train} } &1e-3; 1e-8& 01:09 & 9.9-5 $|$ 9.5-5 $|$ 9.4-5 $|$ 8.2-5& 29:09 $|$ 09:19 $|$ 05:14 $|$ 01:03:59\\ &1e-4; 1e-7& 05:08 & 1.0-4 $|$ 9.9-5 $|$ 9.9-5 $|$ 5.4-5& 35:50 $|$ 14:44 $|$ 11:13 $|$ 01:47:40\\\hline \multirow{ 2}*{\tabincell{c}{log1p.test} } &1e-3; 1e-8& 01:06 & 9.7-5 $|$ 9.5-5 $|$ 9.3-5 $|$ 2.6-5& 08:20 $|$ 08:01 $|$ 03:36 $|$ 49:14\\ &1e-4; 1e-7& 04:26 & 1.0-4 $|$ 1.0-4 $|$ 9.9-5 $|$ 9.6-5& 13:07 $|$ 12:37 $|$ 09:31 $|$ 01:30:30\\\hline \multirow{ 2}*{\tabincell{c}{pyrim5} } &1e-3; 1e-6& 21 & 1.0-4 $|$ 1.0-4 $|$ 1.0-4 $|$ 9.7-5& 06:17 $|$ 14:19 $|$ 14:52 $|$ 34:08\\ &1e-4; 1e-6& 28 & 1.0-4 $|$ 1.0-4 $|$ 1.0-4 $|$ \red{2.1-4}& 05:50 $|$ 21:17 $|$ 12:37 $|$ 39:31\\\hline \multirow{ 2}*{\tabincell{c}{triazines4} } &1e-2; 1e-6& 56 & 1.0-4 $|$ 1.0-4 $|$ 1.0-4 $|$ 9.6-5& 55:38 $|$ 01:47:38 $|$ 02:01:36 $|$ 02:25:28\\ &1e-3; 1e-6& 02:00 & 1.0-4 $|$ \red{4.7-4} $|$ \red{9.1-4} $|$ \red{3.1-3}& 01:12:35 $|$ 03:00:01 $|$ 01:34:51 $|$ 03:00:00\\\hline \multirow{ 2}*{\tabincell{c}{abalone} } &1e-3; 1e-5& 01 & 1.0-4 $|$ 1.0-4 $|$ 1.0-4 $|$ 7.5-5& 32 $|$ 24 $|$ 08 $|$ 40\\ &1e-4; 1e-5& 05 & 9.9-5 $|$ 1.0-4 $|$ 1.0-4 $|$ 9.9-5& 01:15 $|$ 01:40 $|$ 37 $|$ 02:00\\\hline \multirow{ 2}*{\tabincell{c}{bodyfat} } &1e-5; 1e-6& 06 & 9.9-5 $|$ 9.9-5 $|$ 9.9-5 $|$ 9.5-5& 53 $|$ 02:59 $|$ 29 $|$ 07:57\\ &1e-6; 1e-6& 11 & 1.0-4 $|$ 1.0-4 $|$ 1.0-4 $|$ 9.8-5& 01:18 $|$ 05:41 $|$ 01:11 $|$ 11:12\\\hline \multirow{ 2}*{\tabincell{c}{housing} } &1e-3; 1e-6& 08 & 1.0-4 $|$ 1.0-4 $|$ 1.0-4 $|$ 9.5-5& 01:08 $|$ 04:27 $|$ 03:53 $|$ 06:26\\ &1e-4; 1e-6& 22 & 1.0-4 $|$ 1.0-4 $|$ \red{9.9-3} $|$ \red{3.8-4}& 02:06 $|$ 17:25 $|$ 08:06 $|$ 14:52\\\hline 
	\end{longtable}
\end{center}

\vspace{-1cm}

\begin{center}\scriptsize
	\renewcommand{\multirowsetup}{\centering}
	\renewcommand\arraystretch{1.3}
	\begin{longtable}{|c|c|c|c|c|} 	
		\caption{Same in Table \ref{table_UCI_com4} but we terminate the first-order algorithms when $\eta_{rel}\leq 10^{-6}$. Time is shown in the format of (hours:minutes:seconds). } \label{table_UCI_com6}\\
		\hline
		& &  &  $\eta_{rel}$ & time\\
		\hline
		proname   & $\alpha_1$; $\alpha_2$     &  $t_{\mbox{\sc Ssnal}}$  & b $|$ c $|$ d $|$ e & b $|$ c $|$ d $|$ e \\
		\hline
		\endfirsthead
		
		\multicolumn{5}{c}{{\bfseries \tablename\ \thetable{} -- continued from previous page}} \\
		
		\hline
		& &  & $\eta_{rel}$ & time\\
		\hline
		proname   & $\alpha_1$; $\alpha_2$     &  $t_{\mbox{\sc Ssnal}}$  & b $|$ c $|$ d $|$ e & b $|$ c $|$ d $|$ e \\
		\hline
		\endhead
		
		\hline
		\multicolumn{5}{|r|}{{Continued on next page}} \\
		\hline
		\endfoot
		
		\hline
		\endlastfoot
		
		\multirow{ 2}*{\tabincell{c}{E2006.train} } &1e-6; 1e-7& 06 & 9.9-7 $|$ 9.9-7 $|$ 9.9-7 $|$ \red{4.5-5}& 37:47 $|$ 08:01 $|$ 48 $|$ 01:00:44\\ &1e-7; 1e-5& 01:20 & 9.8-7 $|$ 9.9-7 $|$ 9.9-7 $|$ \red{5.5-3}& 44:51 $|$ 11:17 $|$ 01:46 $|$ 59:38\\\hline \multirow{ 2}*{\tabincell{c}{E2006.test} } &1e-6; 1e-7& 04 & 9.8-7 $|$ 9.8-7 $|$ 9.9-7 $|$ \red{3.9-3}& 03:59 $|$ 03:26 $|$ 42 $|$ 23:57\\ &1e-7; 1e-5& 04:00 & \red{1.3-5} $|$ \red{2.6-5} $|$ \red{2.1-5} $|$ \red{6.6-2}& 05:19 $|$ 04:43 $|$ 02:06 $|$ 24:21\\\hline \multirow{ 2}*{\tabincell{c}{log1p.train} } &1e-3; 1e-8& 01:09 & 9.7-7 $|$ 9.6-7 $|$ 1.0-6 $|$ 7.5-7& 32:57 $|$ 12:34 $|$ 10:21 $|$ 02:07:44\\ &1e-4; 1e-7& 05:08 & 9.9-7 $|$ 1.0-6 $|$ 1.0-6 $|$ \red{8.0-6}& 57:11 $|$ 33:23 $|$ 34:48 $|$ 03:00:01\\\hline \multirow{ 2}*{\tabincell{c}{log1p.test} } &1e-3; 1e-8& 01:06 & 9.9-7 $|$ 1.0-6 $|$ 7.8-7 $|$ 7.9-7& 12:20 $|$ 11:48 $|$ 05:11 $|$ 02:04:11\\ &1e-4; 1e-7& 04:26 & 9.9-7 $|$ 9.9-7 $|$ 9.8-7 $|$ \red{7.3-6}& 21:21 $|$ 20:11 $|$ 20:41 $|$ 03:00:01\\\hline \multirow{ 2}*{\tabincell{c}{pyrim5} } &1e-3; 1e-6& 21 & 1.0-6 $|$ 1.0-6 $|$ \red{3.5-5} $|$ \red{8.2-5}& 30:08 $|$ 55:12 $|$ 27:44 $|$ 35:11\\ &1e-4; 1e-6& 28 & 1.0-6 $|$ 1.0-6 $|$ \red{7.2-5} $|$ \red{2.1-4}& 36:43 $|$ 01:51:44 $|$ 19:41 $|$ 39:31\\\hline \multirow{ 2}*{\tabincell{c}{triazines4} } &1e-2; 1e-6& 56 & 1.0-6 $|$ \red{1.1-5} $|$ \red{8.0-5} $|$ \red{1.8-5}& 02:14:05 $|$ 03:00:01 $|$ 02:12:43 $|$ 03:00:01\\ &1e-3; 1e-6& 02:00 & \red{6.6-6} $|$ \red{4.7-4} $|$ \red{9.1-4} $|$ \red{3.1-3}& 03:00:00 $|$ 03:00:01 $|$ 01:34:51 $|$ 03:00:00\\\hline \multirow{ 2}*{\tabincell{c}{abalone} } &1e-3; 1e-5& 01 & 9.6-7 $|$ 9.5-7 $|$ 1.0-6 $|$ 5.1-7& 01:02 $|$ 01:05 $|$ 33 $|$ 02:08\\ &1e-4; 1e-5& 05 & 9.8-7 $|$ 9.7-7 $|$ \red{1.0-4} $|$ \red{2.7-6}& 01:37 $|$ 02:27 $|$ 03:19 $|$ 05:18\\\hline \multirow{ 2}*{\tabincell{c}{bodyfat} } &1e-5; 1e-6& 06 & 1.0-6 $|$ 1.0-6 $|$ 1.0-6 $|$ \red{4.3-6}& 01:10 $|$ 04:00 $|$ 01:39 $|$ 15:25\\ &1e-6; 1e-6& 11 & 9.8-7 $|$ 9.7-7 $|$ 1.0-6 $|$ \red{4.3-5}& 01:43 $|$ 09:07 $|$ 04:11 $|$ 15:18\\\hline \multirow{ 2}*{\tabincell{c}{housing} } &1e-3; 1e-6& 08 & 9.9-7 $|$ 9.9-7 $|$ 1.0-6 $|$ \red{1.7-6}& 02:33 $|$ 10:37 $|$ 11:20 $|$ 14:50\\ &1e-4; 1e-6& 22 & 1.0-6 $|$ 1.0-6 $|$ \red{9.9-3} $|$ \red{3.8-4}& 03:39 $|$ 31:41 $|$ 08:06 $|$ 14:52\\\hline 
	\end{longtable}
\end{center}

\begin{figure}[h]
	\centering\includegraphics[width=4 in]{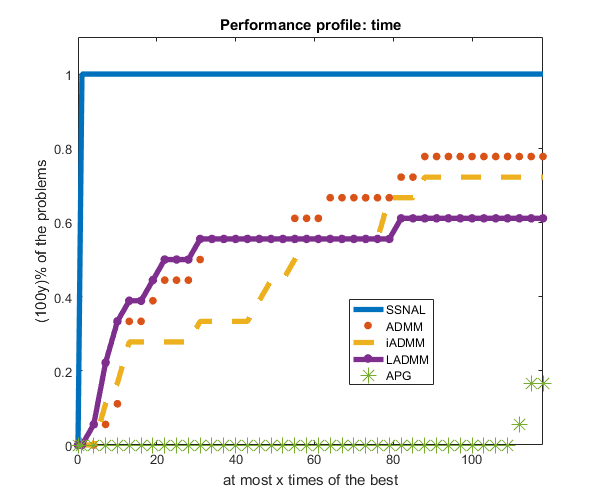}
	\caption{Performance profiles for {\sc Ssnal}, {\sc Admm}, {\sc iAdmm}, {\sc LAdmm} and {\sc Apg} on UCI datasets. Results of {\sc Ssnal} are obtained by setting $\max\{\eta_{gap},\eta_{D},\eta_{kkt}\} \leq 10^{-6}$, and results of {\sc Admm}, {\sc iadmm}, {\sc Ladmm} and {\sc Apg} are obtained by setting $\eta_{rel}\leq 10^{-6}$.}\label{figure_UCI}
\end{figure}

We also present in Figure \ref{figure_UCI} the performance profiles of {\sc Ssnal}, {\sc Admm}, {\sc iAdmm}, {\sc LAdmm} and {\sc Apg} for all the tested problems. In the figure, the results for {\sc Ssnal} are obtained by setting $\max\{\eta_{gap},\eta_{D},\eta_{kkt}\} \leq 10^{-6}$, and the results for {\sc Admm}, {\sc iAdmm}, {\sc LAdmm} and {\sc Apg} are obtained by $\eta_{rel}\leq 10^{-6}$. Thus the accuracy of {\sc Ssnal} is higher than the other algorithms in this sense. Recall that a point $(x,y)$ is in the performance profile curve of a method if and only if it can solve $(100y\%)$ of all tested instances successfully in at most $x$ times of the best methods for each instance. It can be seen that {\sc Ssnal} outperforms all the other methods by a very large margin.

In terms of efficiency and robustness, we can see that {\sc Ssnal} performs much better than all the other first-order methods on these difficult large-scale problem. For example, {\sc Ssnal} only needs about $2$ minutes to produce a solution with the required accuracy such that $\max\{\eta_{gap},\eta_{D},\eta_{kkt}\} \leq 10^{-6}$ for the problem {\rm triazines4}, while all first-order algorithms spend over $1$ hours (3 hours for {\sc iAdmm}) to only produce poor accuracy solutions (with $\eta_{rel}\approx 10^{-4}$) that are much less accurate than {\sc Ssnal}. One can see from Table \ref{table_UCI_com4} and \ref{table_UCI_com6} that {\sc Ssnal} can easily be $5$ to $20$ times faster than the best first-order method on different instances such as {\rm triazines4}, {\rm pyrim5}.

\subsection{Numerical results for synthetic data}
Next we test our algorithms in the high-sample-low-dimension setting. The data used in this subsection are generated randomly from the following true model
\[
b = A x+\varsigma\epsilon,\ \epsilon\sim N({\bf 0}, I).
\]
In the experiments, the rows of $A \in \Re^{m\times n}$ are generated randomly from the multivariate normal distribution $N({\bf 0},\Sigma)$. Here $\Sigma\in \Re^{n\times n}$ is a given symmetric matrix such that $\Sigma_{ij} = \hat{\gamma}^{|i-j|}$ for $i,j=1,\ldots,p$ and $\hat{\gamma}$ is a given parameter. The tuning parameters $\beta$ and $\rho$ in \eqref{given-P} are chosen based on numerical experience.

The examples of $x_0$ presented below were mainly constructed based on the simulation scenarios used in \cite{zou2005regularization, she2010sparse, petry2011pairwise}.  As we want to focus on large-scale problems, we introduce a parameter $k$. In the first six scenarios, we use $k$ to repeat every component of $x_0\in\Re^{n_0}$ by $k$ times consecutively to construct the actual $x\in \Re^{n}$, where $n=n_0 k$, while in the last case, we use $k$ in another strategy which will be explained later. The corresponding number of observations is chosen to be $\max\{80000,0.5nk\}$. We use $80\%$ of the observations to do the training. Instead of using a specified noise level $\varsigma$ for each case, we set $\varsigma = 0.1\|Ax\|/\|\epsilon\|$ for all examples.

\begin{enumerate}
	\item The first setting is specified by the parameter vector
	\[
	x^0\;=\;(3,1.5,0,0,0,2,0,0)^T.
	\]
	The correlation between the $i$-th and $j$-th predictor is
	\[
	{\rm corr} (i,j)=0.9^{|i-j|},\ \forall \ i,j\in\{1,...,8\}.
	\]
	
	\item In this setting, we have $n_0=20$ predictors. The parameter vector is structured into blocks:
	\begin{equation*}
	x^0\;=\;(\underbrace{0,...,0}_5,\underbrace{2,...,2}_5,\underbrace{0,...,0}_5,\underbrace{2,...,2}_5)^T.
	\end{equation*}
	The correlation between $i$-th and $j$-th predictor is given by ${\rm corr}(i,j)=0.3$.
	
	\item This setting consists of $n_0=20$ predictors. The parameter vector is given by
	\begin{equation*}
	x^0\;=\;(5,5,5,2,2,2,10,10,10,\underbrace{0,...,0}_{11})^T.
	\end{equation*}
	Within each of the first three blocks of 3 variables, the correlation between the two predictors  is $0.9$, but there is no correlation among different blocks.
	
	\item The fourth setting consists of $n_0=13$ predictors. The parameter vector is structured into many small clusters:
	\[
	x^0 \;=\; (0,0,-1.5,-1.5,-2,-2,0,0,1,1,4,4,4)^T.
	\]
	The correlation between the $i$-th and $j$-th predictor is
	\[
	{\rm corr} (i,j)=0.5^{|i-j|},\ \forall \ i,j\in\{1,...,13\}.
	\]

	\item The fifth setting is the same as the fourth one, but with a higher correlation between the predictors where ${\rm corr}(i,j)=0.9^{|i-j|},\ \forall \ i,j\in\{1,...,13\}$.
	
	\item In the sixth setting, we have $n_0=16$ predictors. The parameter vectors is structured such that big clusters coexist with small ones:
	\[
	x^0\;=\;(\underbrace{0,...,0}_3,\underbrace{4,...,4}_5,\underbrace{-4,...,-4}_5,2,2,-1)^T.
	\]
	The predictors are possibly negatively correlated: ${\rm corr}(i,j)=(-1)^{|i-j|}0.8$.
	
	\item (Another strategy to use $k$) In the last setting, we use another strategy to construct the example. First we generate a $100$ by $1$ vector $\nu\sim N({\bf 0}, I_{100})$, they we create a histogram bar chart of the elements of vector $\nu$. To be specific, we bin the elements of $\nu$ into $20$ equally spaced containers and return a $20$ by $1$ vector $x_0$ as the number of elements in each container.
	Let
	\[
	x = (\underbrace{x_0,...,x_0}_{2k})^T.
	\]
	The correlation between $i$-th and $j$-th predictor is given by ${\rm corr}(i,j)=0.5$.
\end{enumerate}

For all seven examples, we tested on large-scale problems by setting $k=100$. Note that all instances in this subsection are in the high-sample-low-dimension setting, that is $m\gg n$, just as discussed in Section \ref{sect:ssnalP}, the primal approach is a better choice. In the following text, we use "p-" to represent the primal approach and "d-" to represent the dual approach. For comparison, we terminate all the algorithms when the relative KKT residual $\eta_{kkt}\leq 10^{-6}$.

\begin{figure}[htbp]
	\addtocounter{figure}{1}
	\flushleft
	\subfigure[eg1]{
		\label{fig_eg1}
		\includegraphics[width=0.345\textwidth]{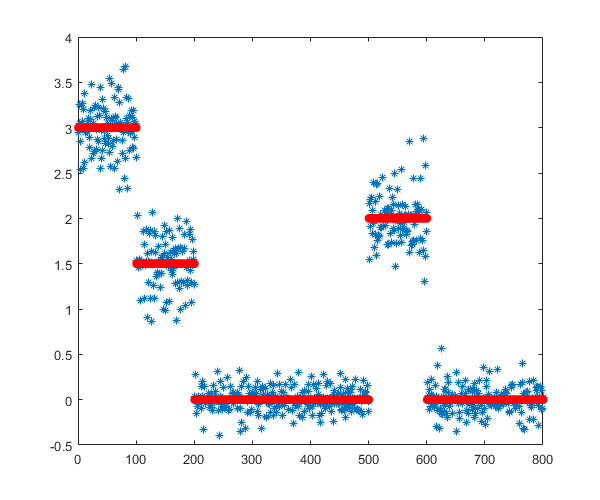}}
	\hspace{-4.5ex}
	\subfigure[eg2]{
		\label{fig_eg2}
		\includegraphics[width=0.345\textwidth]{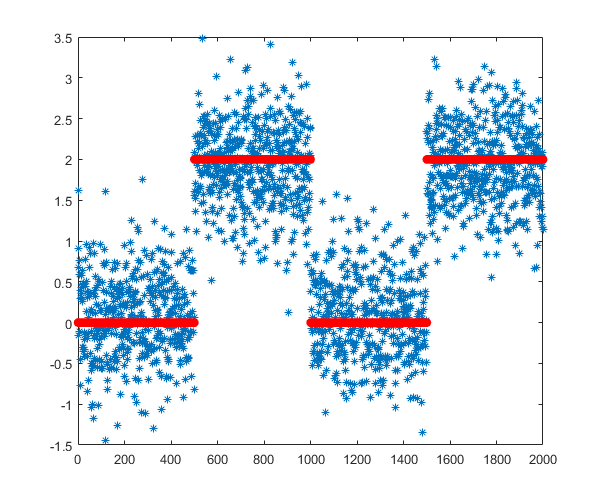}}
	\hspace{-4.5ex}
	\subfigure[eg3]{
		\label{fig_eg3}
		\includegraphics[width=0.345\textwidth]{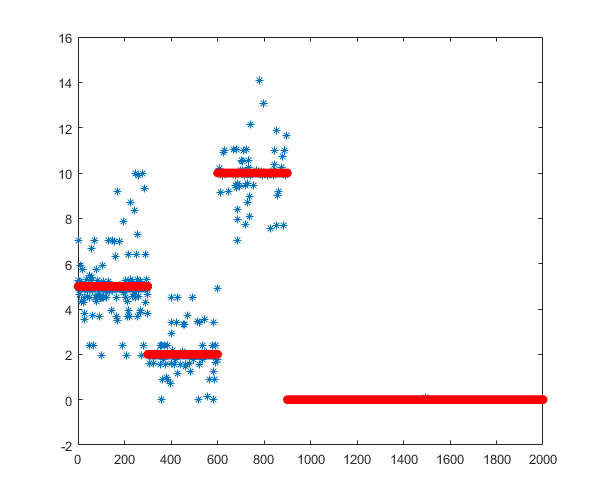}}
	\hspace{-4.5ex}
	\subfigure[eg4]{
		\label{fig_eg4}
		\includegraphics[width=0.345\textwidth]{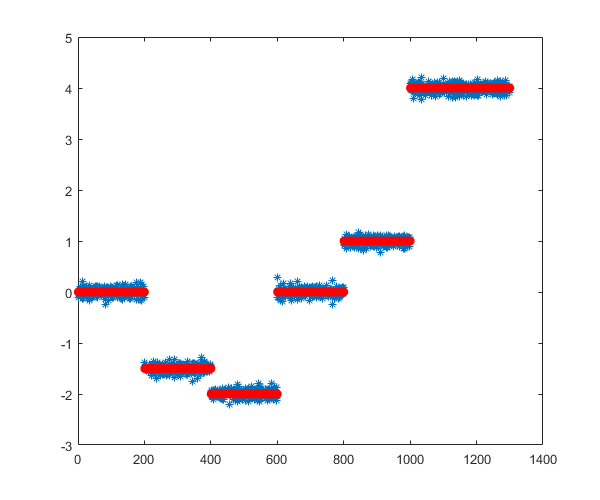}}
	\hspace{-4.5ex}
	\subfigure[eg5]{
		\label{fig_eg5}
		\includegraphics[width=0.345\textwidth]{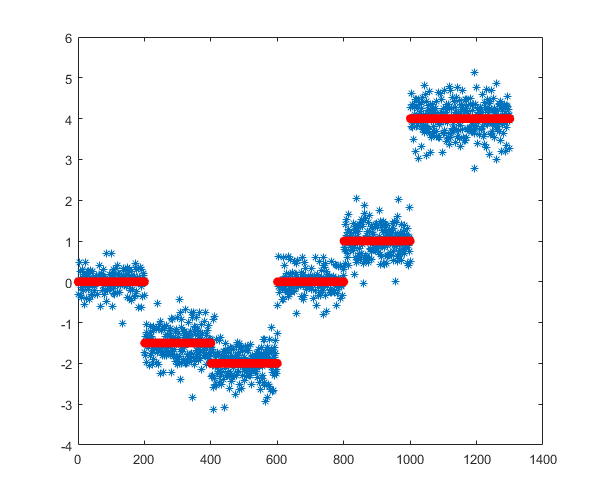}}
	\hspace{-4.5ex}
	\subfigure[eg6]{
		\label{fig_eg6}
		\includegraphics[width=0.345\textwidth]{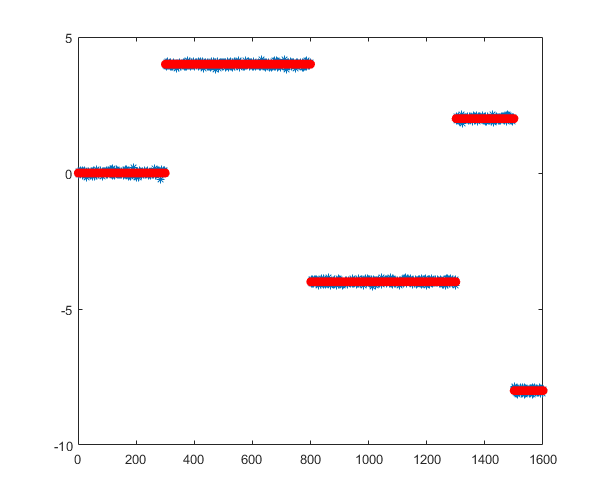}}
	\hspace{-4.5ex}
	\subfigure[eg7]{
		\label{fig_eg7}
		\includegraphics[width=0.345\textwidth]{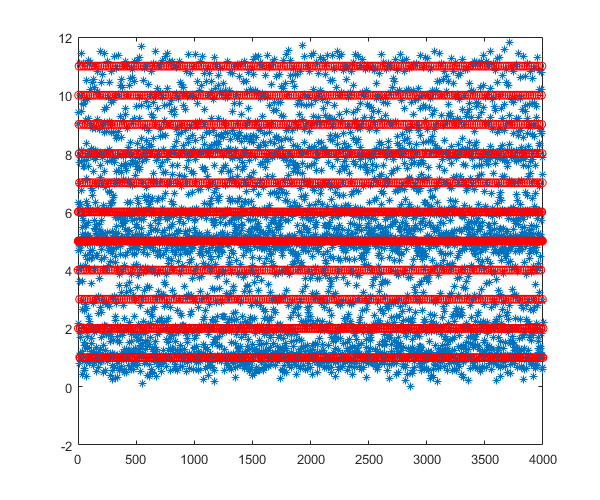}}
	\caption{Recovery results obtained by the p-{\sc Ssnal} algorithm when solving the clustered lasso model on seven synthetic datasets. The algorithm is terminated by setting $\eta_{kkt} \leq 10^{-6}$. In the figures, the red dots represent the actual $x$, and the blue dots represent the solution we obtained by p-{\sc Ssnal}.}
	\label{random_fig}
\end{figure}

Figure \ref{random_fig} shows the recovery results for the seven examples, where the red dots represent the actual $x$, and the blue dots represent the solution we obtained by p-{\sc Ssnal}. As we can see from the figure, the clustered lasso model can recover the group structure of the true regression parameter vector successfully.

As for the purpose of comparing the computational time, we can refer to Table \ref{table_random} for the details. Since d-{\sc iAdmm} and d-{\sc LAdmm} can deal with the case when $m$ is large, we also apply d-{\sc iAdmm} and d-{\sc LAdmm} here. In the table, we can see that p-{\sc Ssnal}, p-{\sc Admm} and d-{\sc LAdmm} can solve all the instances efficiently and accurately. Besides, d-{\sc iAdmm} also gives a good performance except for eg6. The slightly poorer performance of d-{\sc iAdmm} is reasonable since in these cases, $m$ is large that the linear systems needed to be solved in the algorithm are huge. As for p-{\sc Apg}, the numerical results of eg2 and eg6 are not so good since the corresponding Lipschitz constants in the problem are large.

\begin{center}\footnotesize
	\renewcommand{\multirowsetup}{\centering}
	\renewcommand\arraystretch{1.3}
	\begin{longtable}{|c|c|c|}
		\caption{The performance of various algorithms on the synthetic datasets. In the table, "a" = p-{\sc Ssnal}, "b" = p-{\sc Admm}, "c" = p-{\sc Apg}, "d" = d-{\sc iAdmm}, "e" = d-{\sc LAdmm}. We terminate the algorithms when $\eta_{kkt}\leq 10^{-6}$. MSE denotes the MSE obtained by p-{\sc Ssnal}. ${\rm nnz}(x)$ and ${\rm gnnz}(x)$ are obtained by p-{\sc Ssnal}. Time is shown in the format of (hours:minutes:seconds).} \label{table_random} \\
		\hline
		&$\eta_{kkt}$ & time\\
		\hline
		proname   ($m$; $n$) &a $|$ b $|$ c $|$ d $|$ e & a $|$ b $|$ c $|$ d $|$ e \\
		\hline
		\endfirsthead
		
		\multicolumn{3}{c}{{\bfseries \tablename\ \thetable{} -- continued from previous page}} \\
		
		\hline
		&$\eta_{kkt}$ & time \\
		\hline
		proname   ($m$; $n$) &a $|$ b $|$ c $|$ d $|$ e & a $|$ b $|$ c $|$ d $|$ e  \\
		\hline
		\endhead
		
		\hline
		\multicolumn{3}{|r|}{{Continued on next page}} \\
		\hline
		\endfoot
		
		\hline
		\endlastfoot
		
		\multirow{2}*{\tabincell{c}{eg1 (32000; 800)\\ $\lambda_{\max}(A^T A)=$6.21e+05} } &\multicolumn{2}{|c|}{MSE=1.4-1, ${\rm nnz}(x)$=714, ${\rm gnnz}(x)$= 5}\\  \cline{2-3} & 7.4-7 $|$ 9.7-7 $|$ 9.9-7 $|$ 9.9-7 $|$ 9.9-7 & 02 $|$ 01 $|$ 02 $|$ 20 $|$ 44\\\hline \multirow{2}*{\tabincell{c}{eg2 (64000; 2000)\\ $\lambda_{\max}(A^T A)=$3.84e+07} } &\multicolumn{2}{|c|}{MSE=3.4-1, ${\rm nnz}(x)$=1925, ${\rm gnnz}(x)$= 7}\\  \cline{2-3} & 5.2-7 $|$ 9.9-7 $|$ 1.0-6 $|$ 9.0-7 $|$ 4.3-7 & 22 $|$ 01:37 $|$ 16:53 $|$ 42 $|$ 12\\\hline \multirow{2}*{\tabincell{c}{eg3 (64000; 2000)\\ $\lambda_{\max}(A^T A)=$1.76e+07} } &\multicolumn{2}{|c|}{MSE=1.1-1, ${\rm nnz}(x)$=898, ${\rm gnnz}(x)$= 5}\\  \cline{2-3} & 8.0-7 $|$ 9.8-7 $|$ 9.8-7 $|$ 9.2-7 $|$ 9.9-7 & 01:40 $|$ 02:50 $|$ 03:18 $|$ 25 $|$ 01:59\\\hline \multirow{2}*{\tabincell{c}{eg4 (52000; 1300)\\ $\lambda_{\max}(A^T A)=$1.70e+05} } &\multicolumn{2}{|c|}{MSE=3.4-2, ${\rm nnz}(x)$=1285, ${\rm gnnz}(x)$= 6}\\  \cline{2-3} & 6.5-7 $|$ 1.0-6 $|$ 9.6-7 $|$ 9.9-7 $|$ 9.9-7 & 17 $|$ 52 $|$ 09 $|$ 08 $|$ 10\\\hline \multirow{2}*{\tabincell{c}{eg5 (52000; 1300)\\ $\lambda_{\max}(A^T A)=$1.00e+06} } &\multicolumn{2}{|c|}{MSE=1.5-1, ${\rm nnz}(x)$=1231, ${\rm gnnz}(x)$= 5}\\  \cline{2-3} & 8.4-7 $|$ 9.8-7 $|$ 1.0-6 $|$ 9.9-7 $|$ 1.0-6 & 01:26 $|$ 01:37 $|$ 02:07 $|$ 41 $|$ 01:37\\\hline \multirow{2}*{\tabincell{c}{eg6 (64000; 1600)\\ $\lambda_{\max}(A^T A)=$8.19e+07} } &\multicolumn{2}{|c|}{MSE=1.5-2, ${\rm nnz}(x)$=1575, ${\rm gnnz}(x)$= 6}\\  \cline{2-3} & 3.0-7 $|$ 9.9-7 $|$ \red{9.1-3} $|$ \red{9.1-5} $|$ 7.8-7 & 11 $|$ 38 $|$ 21:31 $|$ 35:23 $|$ 11\\\hline \multirow{2}*{\tabincell{c}{eg7 (64000; 4000)\\ $\lambda_{\max}(A^T A)=$2.21e+05} } &\multicolumn{2}{|c|}{MSE=5.5-2, ${\rm nnz}(x)$=3998, ${\rm gnnz}(x)$= 3}\\  \cline{2-3} & 8.1-7 $|$ 9.0-7 $|$ 9.4-7 $|$ 1.0-6 $|$ 9.3-7 & 58 $|$ 01:16 $|$ 50 $|$ 03:22 $|$ 01:00\\\hline 
	\end{longtable}
\end{center}

As we can see, some of the first-order methods are comparable to p-{\sc Ssnal} in these cases in that our new formulation of the clustered lasso regularizer vastly improves the performance of the first-order methods in the projection steps.

\section{Conclusion}
In this paper, we reformulate the clustered lasso regularizer as a weighted ordered-lasso regularizer. Based on the new formulation, we are able to derive a highly efficient algorithm for computing the proximal mapping in $O(n\log (n))$ operations that is crucial for designing efficient first-order and second-order algorithms for solving the clustered lasso problem. Based on efficiently computing the generalized Jacobian of the proximal mapping, we design extremely fast semismooth Newton augmented Lagrangian algorithms, i.e., {\sc Ssnal}, for solving the clustered lasso problem or its dual. Our efficient implementation of the {\sc Ssnal} algorithm heavily relies on the special structure that we have uncovered for the clustered lasso regularizer. The numerical experiments on large-scale real data and synthetic data show the great advantages of our algorithms in comparison with other well designed first-order methods for the clustered lasso problem.

\end{document}